\documentclass[twoside,a4paper,reqno,11pt]{amsart} 
\usepackage[top=28mm,right=30mm,bottom=28mm,left=30mm]{geometry}
\usepackage{framed}
\usepackage{amsfonts, amsmath, amssymb, mathrsfs, bm, latexsym, stmaryrd, array, hyperref, mathtools, bold-extra, xcolor}
\usepackage{booktabs}
\usepackage{enumerate}
\usepackage{comment}

\renewcommand{\a}{\alpha}
\renewcommand{\b}{\beta}
\newcommand{\g}{\gamma}
\newcommand{\e}{\varepsilon}

\newcommand{\normeq}{\trianglelefteqslant}

\newcommand{\la}{\langle}
\newcommand{\ra}{\rangle}

\renewcommand{\to}{\rightarrow}

\newcommand{\leqs}{\leqslant}
\renewcommand{\le}{\leqslant}
\renewcommand{\leq}{\leqslant}
\newcommand{\geqs}{\geqslant}
\renewcommand{\ge}{\geqslant}
\renewcommand{\geq}{\geqslant}

\newcommand{\vs}{\vspace{2mm}}
\newcommand{\fpr}{\mbox{{\rm fpr}}}

\newcommand{\Aut}{\mathrm{Aut}}
\newcommand{\Inn}{\mathrm{Inn}}
\newcommand{\Out}{\mathrm{Out}}
\newcommand{\soc}{\mathrm{soc}}
\newcommand{\PSL}{\mathrm{PSL}}
\newcommand{\GL}{\mathrm{GL}}
\newcommand{\PGL}{\mathrm{PGL}}
\newcommand{\SL}{\mathrm{SL}}

\newcommand{\PGO}{\mathrm{PGO}}

\newcommand{\PGammaL}{\mathrm{P\Gamma L}}
\newcommand{\PSigmaL}{\mathrm{P\Sigma L}}
\newcommand{\val}{\mathrm{val}}

\newcommand{\Sp}{\mathrm{Sp}}

\newcommand{\LL}{\mathrm{L}}
\newcommand{\UU}{\mathrm{U}}

\newcommand{\PGU}{\mathrm{PGU}}

\newcommand{\GU}{\mathrm{GU}}
\newcommand{\reg}{\mathrm{reg}}

\newcommand{\AGL}{\mathrm{AGL}}
\newcommand{\Hol}{\mathrm{Hol}}

\newcommand\sd{\mkern1.5mu{:}\mkern1.5mu}

\definecolor{lblue}{RGB}{50,90,250}

\makeatletter
\newcommand{\imod}[1]{\allowbreak\mkern4mu({\operator@font mod}\,\,#1)}
\makeatother

\newtheorem*{conj*}{Conjecture}

\newtheorem{thm}{Theorem}[section] 
\newtheorem{prop}[thm]{Proposition} 
\newtheorem{lem}[thm]{Lemma}
\newtheorem{cor}[thm]{Corollary} 
\newtheorem{con}[thm]{Conjecture}

\theoremstyle{definition}
\newtheorem{rem}[thm]{Remark}

\newtheorem{defn}[thm]{Definition}

\numberwithin{thm}{section}
\numberwithin{theorem}{section}

\begin{document}
	
	\title[On the generalised Saxl graphs of permutation groups]{On the generalised Saxl graphs of\\ permutation groups} 
	\begin{abstract}
	     A base for a finite permutation group $G \le \mathrm{Sym}(\Omega)$ is a subset of $\Omega$ with trivial pointwise stabiliser in $G$, and the base size of $G$ is the smallest size of a base for $G$. Motivated by the interest in groups of base size two, Burness and Giudici introduced the notion of the Saxl graph. This graph has vertex set $\Omega$, with edges between elements if they form a base for $G$. We define a generalisation of this graph that encodes useful information about $G$ whenever $b(G) \ge 2$: here, the edges are the pairs of elements of $\Omega$ that can be extended to bases of size $b(G)$. In particular, for primitive groups, we investigate the completeness and arc-transitivity of the generalised graph, and the generalisation of Burness and Giudici's Common Neighbour Conjecture on the original Saxl graph.
	\end{abstract}

	\date{\today}
	\author{Saul D. Freedman}
        \address{S.D. Freedman, School of Mathematics, Monash University, Clayton VIC 3800, Australia}
        \curraddr{{\sc Department of Mathematics, Colorado State University, Fort Collins, CO 80523, USA}}
        \email{saul.freedman@colostate.edu}
 
	\author{Hong Yi Huang}
        \address{H.Y. Huang, School of Mathematics, University of Bristol, Bristol BS8 1UG, UK}
        \curraddr{{\sc Department of Mathematics, Southern University of Science and Technology, Shenzhen 518055, Guangdong, P. R. China}}
        \email{11612012@mail.sustech.edu.cn}

	\author{Melissa Lee}
        \address{M. Lee, School of Mathematics, Monash University, Clayton VIC 3800, Australia}
        \email{melissa.lee@monash.edu}
 
	\author{Kamilla Rekvényi}
        \address{K. Rekvényi, Department of Mathematics, University of Manchester, Manchester, M13 9PL, and Heilbronn Institute for Mathematical Research, Bristol, UK}
        \email{kamilla.rekvenyi@manchester.ac.uk}
 
	\maketitle
	
\section{Introduction}

\label{s:intro}

Let $\Omega$ be a finite set, and let $G$ be a group of permutations of $\Omega$, so that $G\leqslant\mathrm{Sym}(\Omega)$. A subset $\Delta \subseteq \Omega$ is called a \emph{base} for $G$ if its pointwise stabiliser is trivial. The smallest cardinality of a base for $G$ is called the \emph{base size} of $G$, denoted $b(G)$. Equivalently, if $G$ is transitive with point stabiliser $H$, then $b(G)$ is the smallest number $b$ such that the intersection of some $b$ conjugates of $H$ in $G$ is trivial. The study of bases has a long history, finding a number of connections to other areas of mathematics, from relational complexity \cite{GLS_RC} to computational group theory \cite{sims}.

Historically, there has been a particular interest in the study of base sizes for primitive groups (recall that a transitive group $G\leqslant\mathrm{Sym}(\Omega)$ is called \emph{primitive} if its point stabilisers are maximal subgroups, or equivalently, the only $G$-invariant partitions of $\Omega$ are $\Omega$ itself and the partition into singletons). In terms of the structure and action of the socle of the group (recall that the \emph{socle} $\soc(X)$ of a group $X$ is the product of its minimal normal subgroups), the finite primitive groups are divided into five families by the O'Nan--Scott theorem. Following \cite{LPS_O'Nan-Scott}, these families are:
\begin{equation*}
    \mbox{affine, almost simple, diagonal type, product type, and twisted wreath products.}
\end{equation*}
Questions related to base sizes have been considered for all five of these families; for example, see \cite{B_sol,BH_prod,F_tw,F_diag,H_diag,L_quasi_1,L_quasi_2}. In particular, 
a significant quantity of research has been conducted on classifying primitive groups of small base size. A project initiated by Saxl in the 1990s to classify primitive base-two groups (i.e.~groups of base size $2$) is one of the most intense research areas in this vein, however the problem remains open in general. 
	
For base-two groups $G\leq \mathrm{Sym}(\Omega)$, Burness and Giudici \cite{BG_Saxl} defined the \emph{Saxl graph} $\Sigma(G)$, whose vertices are the points of $\Omega$, and where two vertices are adjacent if and only if they form a base for $G$. This  graph satisfies a number of interesting properties. For example, if $G$ is transitive, then $\Sigma(G)$ is $G$-vertex-transitive, and if $G$ is primitive, then $\Sigma(G)$ is connected. Burness and Giudici also made a rather striking conjecture that if $G$ is primitive, then any two vertices in $\Sigma(G)$ have a common neighbour (see \cite[Conjecture 4.5]{BG_Saxl}). Moreover, they proved this conjecture in a number of cases, including families of sufficiently large twisted wreath and diagonal type groups, as well as all primitive groups of degree at most 4095. Later on, Burness and Huang \cite{BH_Saxl} proved the conjecture for almost simple primitive groups with soluble point stabilisers. They also proved that the conjecture is valid for the primitive groups with socle $\LL_2(q)$, extending earlier work of Chen and Du \cite{CD_Saxl}. More recently, Lee and Popiel \cite{LP_Saxl} proved the conjecture for most affine type groups with sporadic point stabilisers.

It is evident from the existing literature on base size that a significant number of groups have base size at least $3,$ and for these the Saxl graph is not defined. Hence, in this paper, we generalise the definition of the Saxl graph in order to study groups of arbitrary base size at least $2$, as follows.

\begin{defn}
	Let $G\leqs\mathrm{Sym}(\Omega)$ be a permutation group with $b(G)\geqs 2$. Then the \emph{generalised Saxl graph} of $G$, denoted $\Sigma(G)$, is the graph with vertex set $\Omega$, with two vertices $\alpha$ and $\beta$ adjacent if and only if $\{\alpha,\beta\}$ is a subset of a base for $G$ of size $b(G)$.
\end{defn}

From now on, we will write $\Sigma(G)$ to denote the \emph{generalised Saxl graph} of $G$. Observe that this graph coincides with the Saxl graph when $b(G) = 2$. Note also that our generalisation of the Saxl graph is analogous to the generalisation of the generating graph of a group to the \emph{rank graph} (see \cite{rankgraph}). In Section~\ref{s:irredundant}, we shall consider a different generalisation of the Saxl graph, related to \emph{irredundant bases} for $G$.

A large proportion of this paper is devoted to studying the generalised Saxl graphs $\Sigma(G)$ of primitive groups $G$, in the context of three important problems, which we now describe in detail.

\subsection{Common Neighbour Conjecture}

Our generalisation of the definition of $\Sigma(G)$ yields the following natural generalisation of the aforementioned conjecture of Burness and Giudici.

\begin{con}[Common Neighbour Conjecture]
	\label{conj:BG}
	Let $G$ be a primitive permutation group with $b(G)\geqs 2$. Then any two vertices of $\Sigma(G)$ have a common neighbour.
\end{con}

Notice that if $b(G)\geqs 3$, then the definition of $\Sigma(G)$ implies that any two adjacent vertices have a common neighbour. Therefore, the above conjecture holds for $G$ if and only if $\Sigma(G)$ has diameter at most two.

In this paper, we verify this conjecture for a number of cases. In Section~\ref{s:as} we consider the case where $G$ is an almost simple group, and prove the following theorem. 

\begin{thm}
\label{thm:asconj}
    Let $G$ be an almost simple primitive group. Further suppose that either:
    \begin{itemize}\addtolength{\itemsep}{0.2\baselineskip}
        \item[{\rm (i)}] $b(G) \ge 3$ and $\soc(G)$ is sporadic; or
        
        \item[{\rm (ii)}] $G$ has soluble point stabilisers.
    \end{itemize} 
Then Conjecture \ref{conj:BG} holds.
\end{thm}

Note also that, together with \cite[Theorem 4.22]{BH_Saxl} on base-two groups, Theorem~\ref{thm:PSL2_semi} below directly implies that Conjecture \ref{conj:BG} holds for all almost simple primitive groups with $\soc(G) = \LL_2(q)$ with $q$ a prime power, other than a difficult case where the point stabilisers have a particular structure, see Theorem \ref{thm:PSL2_semi}(ii). We leave open the problem of proving the conjecture in that case. Since that case only occurs when $G > \soc(G)$, Conjecture \ref{conj:BG} holds for all primitive actions of $\LL_2(q).$

In Section~\ref{s:affine}, we further list some examples of primitive affine groups that satisfy Conjecture \ref{conj:BG}, and discuss additional ideas that we hope will inspire further research on this conjecture. Finally, in Section~\ref{s:wr}, we investigate primitive groups of product type, and in particular prove that Conjecture \ref{conj:BG} is equivalent to a certain condition on orbits of point stabilisers. For diagonal type groups and twisted wreath products, analogous work has been undertaken by the second author \cite[Theorems 5.6 \& 5.9.1]{H_thesis}. 

\subsection{Complete graphs}
It turns out that in many instances, the graph $\Sigma(G)$ not only satisfies the Common Neighbour Conjecture, but is also complete. 

Note that if $G$ is transitive and $b(G) = 2$, then $\Sigma(G)$ is complete if and only if $G$ is a Frobenius group. However, determining the completeness of $\Sigma(G)$ when $b(G)>2$ is not as straightforward. With this in mind, we introduce the following definition.

\begin{defn}
A permutation group $G\leqs\mathrm{Sym}(\Omega)$ with $b(G)\geqs 2$ is \emph{semi-Frobenius} if any two elements of $\Omega$ lie in a common base for $G$ of size $b(G)$.
\end{defn}

Thus $\Sigma(G)$ is complete if and only if $G$ is semi-Frobenius, and $G$ is Frobenius if and only if $G$ is semi-Frobenius and $b(G) = 2$. As a straightforward example, the natural action of $\GL_n(q)$ on the set of non-zero vectors of $\mathbb{F}_q^n$ has base size $n$, and is semi-Frobenius if and only if $q = 2$. We also note that any $2$-transitive permutation group is semi-Frobenius.

We prove the following theorem for diagonal type primitive groups in Section~\ref{ss:diag_semi-Frob}.

\begin{thm}
\label{thm:diag_semi}
    Let $G$ be a diagonal type primitive group with socle $T^k$ and top group $P\leq S_k$. If $k = 2$, then $G$ is semi-Frobenius if one of the following holds:
        \begin{itemize}\addtolength{\itemsep}{0.2\baselineskip}
            \item[{\rm (i)}] $P = 1$;
            \item[{\rm (ii)}] $T = A_n$ for $n\geqs 7$; or
            \item[{\rm (iii)}] $T$ is a sporadic group.
        \end{itemize}
        If instead $k\geq 3,$ then the following statements hold.
    \begin{itemize}\addtolength{\itemsep}{0.2\baselineskip}
        \item[{\rm (i)}] If $P\notin\{A_k,S_k\}$ then $b(G) = 2$ and $G$ is not semi-Frobenius.
        \item[{\rm (ii)}] Suppose $\vert T\vert ^{\ell-1}\leq k \leq \vert T\vert ^\ell$ with $\ell\geq 1.$ If $b(G)=\ell+2,$ then $G$ is semi-Frobenius.
        \item[{\rm (iii)}] Suppose $P\in \{A_k,S_k\}$ and $\vert T\vert ^{\ell-1}+3\leq k \leq \vert T\vert ^\ell$ for some $\ell\geq 1$ with $b(G)=\ell+1.$ Then $G$ is not semi-Frobenius.
    \end{itemize}
\end{thm}

We note that the only remaining cases for a complete classification of semi-Frobenius primitive groups of diagonal type are the case where $k\in\{|T|^{\ell-1}+1,|T|^{\ell-1}+2\}$ for $\ell\geqs 2$, and the case where $P = S_2$. In the former case, a partial result is given in Lemma \ref{l:diag_k=+1+2_Sk_incomplete}, which shows that $G$ is not semi-Frobenius if $S_k\leqs G$ (this implies that $P = S_k$, but the reverse implication does not hold). For the case where $P = S_2,$ we show that $G$ is not semi-Frobenius if $G = T^2.(\Out(T)\times S_2)$ with $T = \LL_2(q)$ and $q\geqs 11$ (see Proposition \ref{p:diag_star_psl2}).

Additionally, in Section~\ref{s:as}, we provide a partial classification of semi-Frobenius almost simple primitive groups with sporadic socle. We further give a complete classification for groups whose socle is $\LL_2(q)$, as summarised in the following theorem. Here and throughout the paper, we follow the notation of \cite{B_class,kleidmanliebeck} when referring to the \emph{type} of a maximal subgroup of a classical group, which provides a rough description of its group structure.
\begin{thm}\label{thm:PSL2_semi}
    Let $G$ be an almost simple primitive group with socle $G_0=\LL_2(q)$ and point stabiliser $H.$ Then $G$ is not semi-Frobenius if and only if either:
    \begin{itemize}\addtolength{\itemsep}{0.2\baselineskip}
        \item[{\rm (i)}] $b(G)=2$; or
        \item[{\rm (ii)}] $H$ is of type $\GL_2(q^{1/2})$ with $q\geq 16$ and $\vert G:G_0\vert$ even.
    \end{itemize}
\end{thm}

We refer the reader to Theorem \ref{t:PSL2_semi} for a more detailed statement of Theorem \ref{thm:PSL2_semi}. A key step in our proof of this result is Theorem~\ref{t:PSL2_b(G)}, which extends results in \cite{B_sol,B_class} by determining the exact base size of every primitive group with socle $\LL_2(q)$.

\subsection{Regular orbits and arc-transitivity}
\label{subsec:regintro}

Suppose that $G$ is transitive, and write $\reg(G)$  for the number of regular orbits in the componentwise action of $G$ on $\Omega^{b(G)}$. It is clear that in this case $\Sigma(G)$ is $G$-vertex-transitive since $G$ itself is transitive on $\Omega$. By the Orbit-Stabiliser Theorem, the ordered bases for $G$ of minimal size are precisely the elements of $\Omega^{b(G)}$ lying in regular orbits. Hence, $\Sigma(G)$ is $G$-arc-transitive if $\reg(G) = 1$, as in this case all ordered bases of minimal size lie in the same $G$-orbit. The converse statement, that $\reg(G) = 1$ if $\Sigma(G)$ is $G$-arc-transitive, clearly holds if $b(G) = 2$, but not in general if $b(G) > 2$. For example, if $G$ is the primitive group $\LL_3(4)$ of degree $56$, then $b(G) = 3$ and $\Sigma(G)$ is $G$-arc-transitive, while $\reg(G) = 4$.

Our aim hence is to classify primitive groups $G$ with $\reg(G)=1$ and, more broadly, primitive groups $G$ such that $\Sigma(G)$ is $G$-arc-transitive.
In this direction, Burness and Huang determined the almost simple primitive groups $G$ with soluble stabiliser and $\reg(G) = 1$ (see \cite[Theorem 1]{BH_prod}). Furthermore, the diagonal type primitive groups $G$ with $b(G) = 2$ and $\reg(G) = 1$ were determined in \cite[Theorem 2]{H_diag}. In Section~\ref{ss:diag_arc}, we generalise that result as follows, and show that the groups $G$ from \cite[Theorem 2]{H_diag} with $b(G) = 2$ are the only diagonal type groups for which $\Sigma(G)$ is $G$-arc-transitive.

\begin{thm}
	\label{thm:diag_arc-trans}
	Let $G$ be a diagonal type primitive group with socle $T^k$. Then $\Sigma(G)$ is $G$-arc-transitive if and only if $T = A_5$, $k\in\{3,57\}$ and $G = T^k.(\Out(T)\times S_k)$, so that $b(G) = 2$.
\end{thm}

Thus we obtain the following corollary. 

\begin{cor}
	\label{cor:diag_r=1}
	Let $G$ be a diagonal type primitive group with socle $T^k$. Then $\reg(G) = 1$ if and only if $T = A_5$, $k\in\{3,57\}$ and $G = T^k.(\Out(T)\times S_k)$, so that $b(G) = 2$.
\end{cor}

Additionally, in Section~\ref{subsec:l2q}, we are able to classify the almost simple primitive groups $G$ with socle $\LL_2(q)$ such that $\Sigma(G)$ is $G$-arc-transitive, excluding one difficult case.

\begin{thm}
    \label{introthm:PSL2_arc}
    Let $G$ be a primitive group with socle $G_0 = \LL_2(q)$ and point stabiliser $H$, and if $H$ is of type $\GL_2(q^{1/2})$ with $q \ge 16$, then assume that $|G:G_0|$ is odd. Then $\Sigma(G)$ is $G$-arc-transitive if and only if one of the following holds:
    \begin{itemize}\addtolength{\itemsep}{0.2\baselineskip}
		\item[{\rm (i)}] $H$ is of type $P_1$;
        \item[{\rm (ii)}] $G$ is $2$-transitive and $(G,H) = (\LL_2(11),A_5)$, $(\PGammaL_2(8),D_{18}.3)$, $(\LL_2(7),S_4)$, $(S_6,S_5)$, $(A_6,A_5)$, $(S_5,S_4)$ or $(A_5,A_4)$.
		\item[{\rm (iii)}] $b(G) = 2$, $G = \PGL_2(q)$, $H$ has type $\GL_1(q)\wr S_2$, and $5 \ne q\geqs 4$; or
		\item[{\rm (iv)}] $b(G) = 2$ and $(G,H) = (\LL_2(29),A_5)$, $(\PGL_2(11),S_4)$, $(\mathrm{M}_{10},5{:}4)$  or $(A_5,S_3)$.
	\end{itemize}
    In particular, $\reg(G) = 1$ if and only if $G$ is one of the groups in parts (iii) and (iv), or $H$ is of type $P_1$ and $G$ is sharply $3$-transitive, or $(G,H) \in \{(S_6,S_5), (A_6,A_5),(S_5,S_4)\}$.
\end{thm}

Finally, in Section~\ref{s:wr}, we investigate whether primitive groups $G$ of product type could satisfy $\reg(G)=1.$

\subsection*{Organisation}

We begin our study of the generalised Saxl graph in Section~\ref{s:pre} by noting several basic observations and probabilistic results. We then consider the above problems on primitive permutation groups, for four of the five O'Nan--Scott families: diagonal type groups in Section~\ref{s:diag}, almost simple groups in Section~\ref{s:as}, affine groups in Section~\ref{s:affine}, and finally product type groups in Section~\ref{s:wr}. In Section~\ref{s:irredundant}, we briefly explore an alternative generalisation of the Saxl graph.
\subsection*{Acknowledgements}

The first and fourth authors were partially supported by EPSRC grant numbers EP/W522422/1 and EP/W522673/1, respectively. The second author thanks the China Scholarship Council for
supporting his doctoral studies at the University of Bristol. The third author acknowledges the support of an Australian Research Council Discovery Early Career Researcher Award (project number DE230100579).

\section{Preliminaries}

\label{s:pre}

\subsection{Basic observations}

\label{ss:pre_obs}
Let $G\leqs\mathrm{Sym}(\Omega)$ be a finite permutation group with base size $b(G)$ and generalised Saxl graph $\Sigma(G)$. We begin with some elementary observations. Note that $G$ is a subgroup of the automorphism group of $\Sigma(G)$, and if $G$ is transitive, then $\Sigma(G)$ is $G$-vertex-transitive.
\begin{lem}
	\label{l:pre_complete_2-trans}
	Suppose that $G$ is $2$-transitive, with $b(G) \ge 2$. Then $\Sigma(G)$ is complete.
\end{lem}

\begin{proof}
	Note that $\Sigma(G)$ is clearly non-empty. The $2$-transitivity of $G$ therefore implies that $\{\a,\b\}$ is an edge for all $\a,\b\in\Omega$.
\end{proof}

Note also that if $b(G)\geqs 2,$ then $\Sigma(G)$ is complete and $G$-arc-transitive if and only if $G$ is $2$-transitive.

\begin{lem}
    Suppose that $G$ is transitive with $b(G) \ge 3$, and let $\alpha \in \Omega$. Then $\Sigma(G)$ is complete if and only if $\a$ is the unique isolated vertex of $\Sigma(G_\alpha)$.
\end{lem}
\begin{proof}
    Observe that a subset $\Gamma$ of $\Omega \setminus \{\alpha\}$ is a base for $H:=G_\alpha$ of size $b(H) = b(G)-1$ if and only if $\Gamma \cup \{\a\}$ is a base for $G$ of size $b(G)$. Hence the non-isolated vertices of $\Sigma(H)$ are precisely the neighbours of $\alpha$ in $\Sigma(G)$. The transitivity of $G$ now yields the result.
\end{proof}

A graph is called \emph{locally faithful} if the stabiliser of any vertex acts faithfully on its neighbours.

\begin{lem}
	\label{l:pre_loc_faith}
	Suppose that $G$ is transitive with $b(G)\geqs 2$. Then $\Sigma(G)$ is locally faithful.
\end{lem}

\begin{proof}
	Let $\a\in\Omega$, and let $N(\a)$ be the set of neighbours of $\a$ in $\Sigma(G)$. Then there exist $\a_2,\dots,\a_{b(G)}$ such that $\{\a,\a_2,\dots,\a_{b(G)}\}$ is a base for $G$. In particular, each $\a_i$ is adjacent to $\a$ in $\Sigma(G)$. This implies that
	\begin{equation*}
	(G_\a)_{(N(\a))} \leqs (G_\a)_{\a_2,\dots,\a_{b(G)}} = G_{\a,\a_2,\dots,\a_{b(G)}} = 1,
	\end{equation*}
	and the result follows from the vertex-transitivity of $\Sigma(G)$.
\end{proof}

We now point out elementary relationships between the generalised Saxl graph of $G\leq \mathrm{Sym}(\Omega)$ and its orbital graphs. Recall that the orbitals of $G$ are the orbits of $G$ on unordered pairs of elements of $\Omega$. An orbital of $G$ is \emph{diagonal} if each of its elements is of the form $(\alpha,\alpha)$ with $\alpha \in \Omega$, and if $G$ is transitive, then it has a unique diagonal orbital. For a non-diagonal orbital $\Gamma$, we define the corresponding orbital graph to be the graph with vertex set $\Omega$ and edge set $\Gamma.$ Note that this is an edge-transitive graph by definition. 

\begin{lem}\label{lem:orbgraphs}
    The graph $\Sigma(G)$ is a union of orbital graphs of $G.$
\end{lem}
\begin{proof}
  Note that if a pair $\{\alpha,\beta\}$ of elements of $\Omega$ is contained in a base for $G$, then each pair in the orbital $\{\alpha,\beta\}^G$ is contained in a base. The result follows. 
\end{proof}
 
By \cite{higmanorb}, all orbital graphs of $G$ are connected if and only if $G$ is primitive on $\Omega.$ Hence we can define the \emph{orbital diameter} of a primitive group $G$ to be the supremum of the diameters of its orbital graphs.
 
 \begin{lem} \label{orbdiam2}
    Suppose that $G$ is primitive. Then $\Sigma(G)$ is connected. Additionally, if $b(G) \ge 3$ and the orbital diameter of $G$ is at most two, then Conjecture~\ref{conj:BG} holds.
\end{lem}

\begin{proof}
    Each connected component of $\Sigma(G)$ is a block for $G$, so the primitivity of $G$ implies that $\Sigma(G)$ is connected. Now,  $\Sigma(G)$ is a union of orbital graphs by Lemma \ref{lem:orbgraphs}, and since $G$ is primitive, each of its orbital graphs is connected \cite[Theorem 3.2A]{DM} (this again shows that $\Sigma(G)$ is connected). Hence, if the diameter of each orbital graph is at most $2,$ then by the transitivity of $G$, the diameter of $\Sigma(G)$ is also at most $2.$ Since $b(G)\geq 3,$ any two adjacent vertices have a common neighbour, and so Conjecture~\ref{conj:BG} holds in this case.
\end{proof}
\begin{rem}
We note that Lemma \ref{orbdiam2} does not always hold for imprimitive groups, as illustrated by the following examples.
\begin{itemize}\addtolength{\itemsep}{0.2\baselineskip}
   \item[{\rm (i)}] Let $G = \GL(V)$ acting on non-zero vectors, where $V$ has dimension at least two and is defined over a field containing at least three elements. Then $b(G) = \dim(V)$ and $\Sigma(G)$ is a complete multipartite graph, with each part being a $1$-space. Even though $G$ is imprimitive, it is worth noting that $\Sigma(G)$ still satisfies the property that any two vertices have a common neighbour.
   \item[{\rm (ii)}] Let $G = \PGL_2(7)$ and consider its faithful transitive action on a set $\Omega$ of size $14$. Then $G_\a\cong S_4$ and a computation using {\sc Magma} \cite{magma} shows that $b(G) = 3$, and that $\Sigma(G)$ has two connected components. For each positive integer $k$, the base size of $G^k$ with its product action on $\Omega^k$ is also $3$, and we see that $\Sigma(G^k)$ is a direct product of $k$ copies of $\Sigma(G)$. In particular, $\Sigma(G^k)$ has $2^k$ connected components, and so the number of connected components in a generalised Saxl graph can be arbitrarily large.
    \item[{\rm (iii)}] Let $L$ be the primitive group labelled (27,8) in {\sc Magma}'s Primitive  Groups Library \cite{primgps}, so that $L = 3^3{:}(2\times S_4)$. Let $W := L\wr S_2$ be an imprimitive group of degree $54$. With the aid of {\sc Magma}, one can check that $W$ contains (imprimitive) transitive subgroups $G_1$ of order $5832$ and $G_2$ of order $34992$ such that $b(G_1) = 3$, $b(G_2) = 4$ and $\mathrm{diam}(\Sigma(G_1)) = \mathrm{diam}(\Sigma(G_2)) = 3$. Thus, the statement of Conjecture \ref{conj:BG} does not hold if we remove the condition of primitivity (even for groups with connected generalised Saxl graphs).
\end{itemize}
\end{rem}
	
\subsection{Probabilistic methods}
\label{ss:pre_prob}

Probabilistic methods were first used to bound the base size of a permutation group by Liebeck and Shalev \cite{LS_prob}, and such methods have since had broad application in the area (see e.g. \cite{B_class,BOW_spor}). In this approach, for a permutation group $G \leq \mathrm{Sym}(\Omega)$ and integer $k\geq 2$, one aims to bound the probability
\[
Q(G,k) = \frac{|\{(\a_1,\dots,\a_k)\in\Omega^k:\bigcap G_{\a_i} \neq 1\}|}{|\Omega|^k}
\]
that a random $k$-tuple of elements from $\Omega$ does not form a base for $G$. If it can be shown that $Q(G,k)<1$, then there must be a base of size $k$ for $G$.

Burness and Giudici \cite[\S 3.3]{BG_Saxl} observed that for transitive groups with base size 2, there is a link between $Q(G,2)$ and the valency $\mathrm{val}(G)$ of the corresponding Saxl graph $\Sigma(G)$. We now generalise this connection to transitive groups of arbitrary base size $b(G)\geq 2$.

\begin{prop}
\label{prop:val}
Let $G \leq \mathrm{Sym}(\Omega)$ be a transitive group with $k := b(G)\geqs 2$. Then
\begin{equation*}
1-Q(G,k)  \le \frac{{\rm{val}}(G)^{k-1}}{|\Omega|^{k-1}}
= \left( \frac{{\rm{val}}(G)}{|\Omega|} \right)^{k-1}.
\end{equation*}
\end{prop}
\begin{proof}
Since $G$ acts transitively, every element of $\Omega$ is contained in the same number of bases of size $k=b(G)$. Hence 
\begin{align*}
1-Q(G,k)  & =   \frac{|\{(\a_1,\dots,\a_k)\in\Omega^k:\bigcap G_{\a_i} = 1\}|}{|\Omega|^k}\\
& = \frac{|\{(\a_2,\dots,\a_k)\in\Omega^{k-1}:\bigcap_{i=1}^{k} G_{\a_i} = 1\}|}{|\Omega|^{k-1}},
\end{align*}
with $\alpha_1 \in \Omega$ fixed in this final expression.
Now, the set of elements of $\Omega$ that appear in a tuple in $\{(\a_2,\dots,\a_k)\in\Omega^{k-1}:\bigcap_{i=1}^{k} G_{\a_i} = 1\}$ has 
size ${\rm{val}}(G)$.
Therefore, 
\begin{equation*}
1-Q(G,k)  \le \frac{{\rm{val}}(G)^{k-1}}{|\Omega|^{k-1}}
= \left( \frac{{\rm{val}}(G)}{|\Omega|} \right)^{k-1}. \qedhere
\end{equation*}
\end{proof}
This allows us to extend \cite[Lemma 3.6]{BG_Saxl} as follows.

\begin{lem}
	\label{l:pre_prob}
	Let $G$ be a transitive group with $k:=b(G) > 1$, and set
	\begin{equation*}
	t:=\max\{m\in\mathbb{N} : Q(G,k) < 1/m\}.
	\end{equation*}
 Additionally, let $r:=\max\{t,(k-1)(t-1)\}$. If $r \geqs 2$, then the following properties hold.
	\begin{itemize}\addtolength{\itemsep}{0.2\baselineskip}
		\item[{\rm (i)}] Any $r$ vertices in $\Sigma(G)$ have a common neighbour.
		\item[{\rm (ii)}] Every edge in $\Sigma(G)$ is contained in a complete subgraph of size $r+1$.
		\item[{\rm (iii)}] The clique number of $\Sigma(G)$ is at least $r+1$.
		\item[{\rm (iv)}] $\Sigma(G)$ is connected with diameter at most $2$.
		\item[{\rm (v)}] $\Sigma(G)$ is Hamiltonian.
	\end{itemize}
\end{lem}
\begin{proof}
We proceed as in the proof of \cite[Lemma 3.6]{BG_Saxl}. If $Q(G,k) < 1/t$, then $|\Omega|(1-1/t)^{1/(k-1)} < \mathrm{val}(G)$ by Proposition \ref{prop:val}. Now $(1-1/t)^{1/(k-1)}\geq 1-\frac{1}{(k-1)(t-1)}$ since, for $t\geq 2$, 
\[
\left(1-\frac{1}{(k-1)(t-1)}\right)^{k-1} \leq e^{-(t-1)} \leq 1-\frac{1}{t-1}+ \frac{1}{2(t-1)^2}\leq 1-\frac{1}{t}.\]
Hence $\mathrm{val}(G) > |\Omega|(1-\frac{1}{r})$. This implies that, for any $r$ vertices in $\Sigma(G)$, the neighbourhoods of those vertices have a non-trivial common intersection. That is, those vertices have a common neighbour. This gives (i), and since $r \ge 2$, parts (ii) -- (iv) follow directly. Finally, since $\mathrm{val}(G)>|\Omega|/2$, Dirac's Theorem \cite[Theorem 3]{D_hamiltonian} directly yields (v).
\end{proof}

A direct application of Proposition \ref{prop:val} and Lemma \ref{l:pre_prob} yield the following.
\begin{cor}
Let $G$ be a transitive group with $b(G)\geqs 2$. If $Q(G,b(G)) < 1-2^{-b(G)+1}$, then 
${\rm{val}}(G)/|\Omega| > 1/2$. In particular, every pair of vertices in $\Sigma(G)$ has a common neighbour. 
\end{cor}

Now, the probability that a random element in $\Omega$ is fixed by $x\in G$ is given by the \emph{fixed point ratio}
\begin{equation*}
\fpr(x,\Omega) = \frac{|\mathrm{fix}(x,\Omega)|}{|\Omega|},
\end{equation*}
where $\mathrm{fix}(x,\Omega)$ denotes the set of points fixed by $x$. In particular, if $G$ is transitive, then
\begin{equation*}
\fpr(x,\Omega) = \frac{|x^G\cap H|}{|x^G|},
\end{equation*}
where $H$ is a point stabiliser of $G$. A powerful probabilistic approach to study $b(G)$ and $\Sigma(G)$, originally introduced by Liebeck and Shalev \cite{LS_prob}, is to use fixed point ratios to bound $Q(G,k)$ from above. Note that if $\{\a_1,\dots,\a_k\}\subseteq \Omega$ is not a base for $G$, then there exists $x\in G_{\a_1,\dots,\a_k}$ of prime order. Therefore, if $G$ is transitive, then
\begin{equation*}
Q(G,k)\leqs \sum_{x\in R(G)}|x^G|\cdot \fpr(x,\Omega)^k = \sum_{x\in R(G)}\frac{|x^G\cap H|^k}{|x^G|^{k-1}} =: \widehat{Q}(G,k),
\end{equation*}
where $R(G)$ is a set of representatives for the $G$-conjugacy classes of elements in $H$ which have prime order. In particular, if $\widehat{Q}(G,b(G)) < 1/t$ for some integer $t\geqs 2$, then each property listed in Lemma \ref{l:pre_prob} holds for $\Sigma(G)$.
\section{Diagonal type groups}

\label{s:diag}

In this section we consider primitive groups of diagonal type. We first recall their construction and list some preliminary results that we will use later. We then consider completeness and arc-transitivity of the generalised Saxl graphs of these groups.

Here we follow the notation in \cite{F_diag}. Let $k\geqs 2$ be an integer and let $T$ be a non-abelian simple group. Then
\begin{equation*}
    X := \{(t,\dots,t) : t\in T\}
\end{equation*}
is a core-free subgroup of $T^k$, and defining $\Omega := [T^k:X]$, we see that $T^k$ is a transitive subgroup of $\mathrm{Sym}(\Omega)$ (via the usual action on right cosets). Let $W(k,T) := N_{\mathrm{Sym}(\Omega)}(T^k)$, noting that
$$W(k,T) = T^k.(\Out(T)\times S_k)$$
has socle $T^k$. A group $G\leqs\mathrm{Sym}(\Omega)$ is of \emph{diagonal type} if 
\begin{equation*}
T^k\normeq G\leqs W(k,T).
\end{equation*}
Note that $G$ induces a subgroup $P_G$ of $S_k$ on the set of factors of $T^k$ by the conjugation action. More precisely,
\begin{equation*}
P_G = \{\pi\in S_k:(\a_1,\dots,\a_k)\pi\in G\text{ for some }\a_1,\dots,\a_k\in\Aut(T)\}.
\end{equation*}
The group $P_G$ is called the \emph{top group} of $G$, and naturally we have
\begin{equation*}
T^k\normeq G\leqs T^k.(\Out(T)\times P_G).
\end{equation*}

Throughout the rest of this section, we will assume that $G\leqs\mathrm{Sym}(\Omega)$ is a diagonal type group. Additionally, if $G$ is clear from the context, we shall write $P := P_G$ and $W := T^k.(\Out(T)\times P)$. Observe also that
\begin{equation*}
D:=\{(\alpha,\dots,\alpha)\pi:\alpha\in\Aut(T),\pi\in P\}
\end{equation*}
is a point stabiliser of $W$ (note that $W\leqs \Aut(T)\wr S_k$). Thus, we make the identification
\begin{equation*}
    \Omega= [W:D] = \{D(\varphi_{t_1},\dots,\varphi_{t_k}):t_1,\dots,t_k\in T\},
\end{equation*}
where $\varphi_t\in\Inn(T)$ denotes the inner automorphism of $T$ such that $x^{\varphi_t} = t^{-1}xt$ for all $x\in T$.
Moreover, the action of $G$ on $\Omega$ is given by
\begin{equation*}
D(\varphi_{t_1},\dots,\varphi_{t_k})^{(\alpha_1,\dots,\a_k)\pi} = D(\varphi_{t_{1^{\pi^{-1}}}}\a_{1^{\pi^{-1}}},\dots,\varphi_{t_{k^{\pi^{-1}}}}\a_{k^{\pi^{-1}}}),
\end{equation*}
and so the stabiliser of $D\in\Omega$ in $W$ is $D$ itself. Note that $\a^{-1}\varphi_t\a = \varphi_{t^\a}$ for all $t\in T$ and $\alpha \in \Aut(T)$, which implies that
\begin{equation*}
D(\varphi_{t_1},\dots,\varphi_{t_k})^{(\alpha,\dots,\a)\pi} = D(\varphi_{t_{1^{\pi^{-1}}}^{\a}},\dots,\varphi_{t_{k^{\pi^{-1}}}^{\a}})
\end{equation*}
for any element $(\a,\dots,\a)\in D$. 

Moreover, $G$ is primitive if and only if either $P$ is primitive on $[k]:=\{1,2,\ldots,k\}$, or $k = 2$ and $P = 1$.

The precise base sizes of primitive groups of diagonal type have been computed in \cite[Theorem 3]{H_diag} as follows, extending an earlier result of Fawcett \cite{F_diag}.
\begin{thm}{\cite[Theorem 3]{H_diag}}
	\label{t:H_diag}
	Let $G$ be a diagonal type primitive group with socle $T^k$ and top group $P\leqs S_k$.
	\begin{itemize}\addtolength{\itemsep}{0.2\baselineskip}
		\item[{\rm (i)}] If $P\notin\{A_k,S_k\}$, then $b(G) = 2$.
		\item[{\rm (ii)}] If $k = 2$, then $b(G)\in\{3,4\}$, with $b(G) = 4$ if and only if $T\in\{A_5,A_6\}$ and $G = T^2.(\Out(T)\times S_2)$.
		\item[{\rm (iii)}] If $k\geqs 3$, $P\in\{A_k,S_k\}$ and $|T|^{\ell-1} < k\leqs |T|^{\ell}$ with $\ell\geqs 1$, then $b(G)\in\{\ell+1,\ell+2\}$. Moreover, $b(G) = \ell+2$ if and only if one of the following holds:
		\vspace{1mm}
		\begin{itemize}\addtolength{\itemsep}{0.2\baselineskip}
			\item[{\rm (a)}] $k = |T|$.
			\item[{\rm (b)}] $k\in\{|T|-2,|T|^\ell-1,|T|^\ell\}$ and $S_k\leqs G$.
			\item[{\rm (c)}] $k = |T|^2-2$, $T\in\{A_5,A_6\}$ and $G = T^k.(\Out(T)\times S_k)$.
		\end{itemize}
	\end{itemize}
\end{thm}

We will also need the following lemma from \cite{H_diag} for our later analysis.

\begin{lem}{\cite[Corollary 2.6]{H_diag}}
	\label{l:H_diag_c:2.6}
Let $G$ be a diagonal type primitive group with socle $T^k$ and top group $P\leqs S_k$.	If $P\in\{A_k,S_k\}$ and $k\geqs |T|-3$, then $G$ contains $A_k$.
\end{lem}

For any $\mathbf{x} := (\varphi_{t_1},\dots,\varphi_{t_k})\in \Inn(T)^k$, there is an associated partition
\[\mathcal{P}^{\mathbf{x}} := \{\mathcal{P}_t^{\mathbf{x}}:t\in T\}\]
of $[k]$ such that $j\in\mathcal{P}_t^{\mathbf{x}}$ if $t_j = t$, noting that some parts of $\mathcal{P}^{\mathbf{x}}$ might be empty. The following lemma involves elements $(\a,\dots,\a)\pi\in G_D \leqs D$ 
that also lie in other point stabilisers.

\begin{lem}
	\label{l:H_diag_l:2.2}
	Let $\mathbf{x} = (\varphi_{t_1},\dots,\varphi_{t_k})\in \Inn(T)^k$ so that $D\mathbf{x}\in\Omega$, and let $\mathcal{P}^{\mathbf{x}} = \{\mathcal{P}_t^{\mathbf{x}}:t\in T\}$ be the associated partition of $[k]$ as above. Suppose also that $(\a,\dots,\a)\pi\in G_{D\mathbf{x}}$. Then the following properties hold.
	\begin{itemize}\addtolength{\itemsep}{0.2\baselineskip}
		\item[{\rm (i)}] $\pi\in P_{\{\mathcal{P}^{\mathbf{x}}\}}$.
		\item[{\rm (ii)}] Suppose that $0<|\mathcal{P}_1^{\mathbf{x}}|\ne |\mathcal{P}_t^{\mathbf{x}}|$ for all $t \in T \setminus \{1\}$. Then $t_j^\alpha = t_{j^\pi}$ for all $j$.
        \item[{\rm (iii)}] Under the same assumption as (ii), $\a$ lies in the setwise stabiliser of
		\begin{equation*}
		\{g:|\mathcal{P}_g^{\mathbf{x}}| = |\mathcal{P}_t^{\mathbf{x}}|\}
		\end{equation*}
		for any $t\in T$.
	\end{itemize}
\end{lem}

\begin{proof}
    Parts (i) and (ii) follow from \cite[Lemma 2.2]{H_diag}. For part (iii), note that for any $t\in T$, we have $|\mathcal{P}_t^{\mathbf{x}}| = |\mathcal{P}_{t^\a}^{\mathbf{x}}|$ by applying (ii).
\end{proof}

\begin{lem}
	\label{l:diag_m-cycle}
	Let $\pi = (j_1,\dots,j_m)\in S_k$ be an $m$-cycle with $m<k$, and let $\mathbf{x} = (\varphi_{t_1},\dots,\varphi_{t_k})\in\Inn(T)^k$ such that $t_{j_1},\dots,t_{j_m}$ are not all equal. Then $\pi\notin G_D\cap G_{D\mathbf{x}}$.
\end{lem}

\begin{proof}
Suppose for a contradiction that $\pi\in G_D\cap G_{D\mathbf{x}}$. Then there exists $g\in T$ such that $t_j = gt_{j^\pi}$ for all $j\in [k]$. Since $m<k$, there exists $j_0 \in [k] \setminus \{j_1,\dots,j_m\}$, so by taking $j = j_0$ we see that $g = 1$. However, this yields $t_{j_1} = \cdots = t_{j_m}$, which is incompatible with our assumption.
\end{proof}

In the remainder of this section, we shall determine in various cases whether or not $G$ is semi-Frobenius, and then prove Theorem~\ref{thm:diag_arc-trans} by determining precisely when $\Sigma(G)$ is $G$-arc-transitive. 
We also note from Lemma~\ref{orbdiam2} that all primitive groups of diagonal type with orbital diameter two satisfy Conjecture~\ref{conj:BG}. These groups were classified in \cite{kaladiag}.

\subsection{Semi-Frobenius groups}

\label{ss:diag_semi-Frob}

\par

We begin with the case where $k\geq 3.$ The cases that we need to consider are the ones where $b(G)\geq 3,$ which are listed in Theorem \ref{t:H_diag}(iii). We start with a result that characterises a family of semi-Frobenius groups.
\begin{prop}
	\label{p:diag_l+2_complete}
	Suppose that $k\geqs 3$, $P\in\{A_k,S_k\}$, $|T|^{\ell-1}<k\leqs |T|^\ell$ for some $\ell\geqs 1$, and $b(G) = \ell+2$. Then $G$ is semi-Frobenius.
\end{prop}

\begin{proof}
	Since $b(G) = \ell+2$, we see that $G$ is one of the groups listed in cases (iii)(a)--(c) of Theorem \ref{t:H_diag}. Let $\mathbf{x}\in\Inn(T)^k$ so that $D\mathbf{x}\in\Omega\setminus\{D\}$. We need to show that $\{D,D\mathbf{x}\}$ can be extended to a base for $G$ of size $\ell+2$. By Lemma \ref{l:diag_m-cycle}, there exists $(j,j')\in S_k$ such that $(j,j')\notin G_D\cap G_{D\mathbf{x}}$. Hence, without loss of generality, it suffices to find $\mathbf{a}_1,\dots,\mathbf{a}_{\ell}\in \Inn(T)^k$ such that the pointwise stabiliser of $\Delta:=\{D,D\mathbf{a}_1,\dots,D\mathbf{a}_{\ell}\}$ in $T^k.(\Out(T)\times S_k)$ is generated by a transposition in $S_k$. Write $\mathbf{a}_i = (\varphi_{t_{i,1}},\dots,\varphi_{t_{i,k}})$.
	
	First assume that $k\in\{|T|-2,|T|-1\}$ and $S_k\leqs G$, so that $b(G) = 3$ by Theorem \ref{t:H_diag}. Note that \cite[Corollary 5]{H_diag} implies that there exists a subset $S\subseteq T\setminus\{1\}$ such that $|S| = k-2$ and the setwise stabiliser of $S$ in $\Aut(T)$ is trivial. Let $\{t_{1,3},\dots,t_{1,k}\} = S$ and $t_{1,1} = t_{1,2} = 1$. By Lemma \ref{l:H_diag_l:2.2}(iii), if $(\a,\dots,\a)\pi\in G_{D\mathbf{a}_1}$ then $S^\a = S$ and so $\a = 1$, which yields $G_D\cap G_{D\mathbf{a}_1}\leqs\la (1,2)\ra$. Thus, $G$ is semi-Frobenius as noted above.
	
	Next, assume that $k\in\{|T|^\ell-1,|T|^\ell\}$ with $\ell\geqs 2$ and $S_k\leqs G$, or that $k = |T|^2-2$, $T\in\{A_5,A_6\}$ and $G = T^k.(\Out(T)\times S_k)$. In the proof of Lemma 5.10 in \cite{H_diag}, we see that there exists an $(\ell+1)$-subset of $\Omega$ containing $D$ whose pointwise stabiliser in $T^k.(\Out(T)\times S_k)$ is generated by a transposition in $S_k$, and so $G$ is semi-Frobenius by arguing as above.
	
	By Theorem~\ref{t:H_diag}, it remains to assume that $k = |T|$ and $b(G) = 3$. Write $\mathbf{x} = (\varphi_{x_1},\dots,\varphi_{x_k})$. We consider the following two cases in turn and show that $\{D,D\mathbf{x}\}$ can be extended to a base for $G$ of size $3$ in each case.
	
	\vs
	
	\noindent \emph{Case 1. There are at least three distinct entries in $\mathbf{x}$.}
	
	\vs
	
	Without loss of generality, we may assume that $x_1,x_2,x_3$ are distinct. In view of Lemma \ref{l:diag_m-cycle}, it suffices to find $\mathbf{a} = (\varphi_{t_1},\dots,\varphi_{t_k})\in\Inn(T)^k$ such that $G_D\cap G_{D\mathbf{a}}\leqs \la (1,2,3),(1,2)\ra$. Once again, \cite[Corollary 5]{H_diag} implies that there exists a subset $S\subseteq T\setminus\{1\}$ of size $|T|-3$ with trivial setwise stabiliser in $\Aut(T)$. Let $\{t_4,\dots,t_k\} = S$ and let $t_1 = t_2 = t_3 = 1$. Suppose $(\a,\dots,\a)\pi\in G_{D\mathbf{a}}$. Then by Lemma \ref{l:H_diag_l:2.2}(iii) we have $S^\a = S$, which forces $\a = 1$. Thus, $G_D\cap G_{D\mathbf{a}}\leqs \la (1,2,3),(1,2)\ra$ as desired.
	
	\vs

    \noindent \emph{Case 2. There are exactly two distinct entries in $\mathbf{x}$.}

    We may assume that $\mathbf{x} = (1,\dots,1,\varphi_x,\dots,\varphi_x)$ for some non-trivial $x\in T$, with $1$ appearing exactly $m$ times for some $m \le |T|/2$. By the main theorem of \cite{GK_3/2}, there exists $y\in T$ such that $\la x,y\ra = T$. Define $\mathbf{a} = (\varphi_{t_1},\dots,\varphi_{t_k})\in \Inn(T)^k$ such that the following conditions hold:
	\begin{itemize}\addtolength{\itemsep}{0.2\baselineskip}
		\item[{\rm (a)}] $t_m = t_{m+1} = 1$;
		\item[{\rm (b)}] $\{t_1,\dots,t_k\}\setminus\{t_m,t_{m+1}\} = T\setminus\{1,y\}$; and
		\item[{\rm (c)}] if $m = |T|/2$, then some $\Aut(T)$-conjugacy class $C$ is equal to $\{t_1,\dots,t_{|C|}\}$ (note that $|C|<m$).
	\end{itemize}
    Note that since $k = |T|$, each element of $T \setminus \{1,y\}$ is equal to $t_j$ for exactly one $j$. By Lemma~\ref{l:H_diag_l:2.2}(iii), if $(\a,\dots,\a)\pi\in G_{D\mathbf{a}}$, then $\a\in C_{\Aut(T)}(y)$.

    Suppose now that $(\a,\dots,\a)\pi$ also lies in $G_{D\mathbf{x}}$. We claim that $\a \in  C_{\Aut(T)}(x)$. If $m < |T|/2$, then this is immediate from Lemma~\ref{l:H_diag_l:2.2}(iii). Otherwise, the same lemma yields $t_j^\a = t_{j^\pi}$ for all $j\in [k]$ and
	\begin{equation*}
	\{1,\dots,m\}^\pi = \{1,\dots,m\}\mbox{ or }\{m+1,\dots,k\}.
	\end{equation*}
	If the latter holds, then for each $j<m$, there exists $j'\geqs m+1$ such that $t_j^\a = t_{j'}$, contradicting condition (c). Hence, $\{1,\dots,m\}$ is fixed by $\pi$, and $x_j^\a = x_{j^\pi}$ for all $j\in[k]$. This implies that $\a\in C_{\Aut(T)}(x)$, as claimed.

    It now follows that $\a \in C_{\Aut(T)}(\la x, y \ra) = C_{\Aut(T)}(T) = 1$. Moreover, if $x_j = x_{j'}$, then $t_j\ne t_{j'}$, and therefore $\pi = 1$. Thus $\{D,D\mathbf{x},D\mathbf{a}\}$ is a base for $G$.
\end{proof}

The following two results are partial converses to Proposition \ref{p:diag_l+2_complete}.

\begin{prop}
	\label{p:diag_l+1_incomplete}
	Suppose that $P\in\{A_k,S_k\}$, $|T|^{\ell-1}+3\leqs k\leqs |T|^\ell$ for some $\ell\geqs 1$, and $b(G) = \ell+1$. Then $G$ is not semi-Frobenius.
\end{prop}

\begin{proof}
	If $\ell = 1$, then $b(G) = 2$, so $G$ is clearly not semi-Frobenius. Thus, we may assume that $\ell\geqs 2$. Then $T^k \sd A_k\leqs G$ by Lemma \ref{l:H_diag_c:2.6}. Let $\mathbf{x} := (1,\dots,1,\varphi_x)\in \Inn(T)^k$ for some $x\in T \setminus \{1\}$. It suffices to show that $\Delta:=\{D,D\mathbf{x},D\mathbf{a}_1,\dots,D\mathbf{a}_{\ell-1}\}$ is not a base for any choices of $\mathbf{a}_i := (\varphi_{t_{i,1}},\dots,\varphi_{t_{i,k}})\in\Inn(T)^k$.
	
	For each $j \in [k]$, let $\mathbf{c}_j := (t_{1,j},\dots,t_{\ell-1,j})\in T^{\ell-1}$. Note that $k-1\geqs |T|^{\ell-1}+2$. This implies that either there exist $j_1,j_2,j_3 \in [k-1]$ such that $\mathbf{c}_{j_1} = \mathbf{c}_{j_2} = \mathbf{c}_{j_3}$, or there exist distinct $j_1,j_2 \in [k-1]$ and distinct $j_1',j_2' \in [k-1]$ such that $\mathbf{c}_{j_1} = \mathbf{c}_{j_2}\ne \mathbf{c}_{j_1'} = \mathbf{c}_{j_2'}$. In the former case, $(j_1,j_2,j_3)\in G_{(\Delta)}$, while $(j_1,j_2)(j_1',j_2')\in G_{(\Delta)}$ in the latter case.
\end{proof}

In the following lemma, we assume that $S_k \le G$. As we mentioned in Section~\ref{s:intro}, this implies that $P = S_k$, but the reverse implication does not hold.

\begin{lem}
	\label{l:diag_k=+1+2_Sk_incomplete}
	Suppose that $S_k\leqs G$ and $k\in \{|T|^{\ell-1}+1,|T|^{\ell-1}+2\}$ for some $\ell\geqs 2$. Then $G$ is not semi-Frobenius.
\end{lem}

\begin{proof}
	Note that $b(G) = \ell+1$ by Theorem \ref{t:H_diag}. Let $\mathbf{x} := (1,\dots,1,\varphi_x)\in\Inn(T)^k$ for some $x\in T \setminus \{1\}$. Again, it suffices to show that $\Delta:=\{D,D\mathbf{x},D\mathbf{a}_1,\dots,D\mathbf{a}_{\ell-1}\}$ is not a base for any choices of $\mathbf{a}_i := (\varphi_{t_{i,1}},\dots,\varphi_{t_{i,k}})\in\Inn(T)^k$, and without loss of generality, we may assume that $t_{i,k} = 1$ for all $i$. Let $\mathbf{c}_j:=(t_{1,j},\dots,t_{\ell-1,j})\in T^{\ell-1}$ for each $j \in [k]$.
	
	It is easy to see that if $k = |T|^{\ell-1}+2$, then there exist distinct $j,j' \in [k-1]$ such that $\mathbf{c}_j = \mathbf{c}_{j'}$. Thus $(j,j')\in G_{(\Delta)}$, and $\Delta$ is not a base for $G$.
	
	To complete the proof, we assume that $k = |T|^{\ell-1}+1$. Arguing as above, we may also assume that $\mathbf{c}_1,\dots,\mathbf{c}_{k-1}$ are all distinct, otherwise there is a transposition in $G_{(\Delta)}$. In particular, for each $j\in [k-1]$, we have $\mathbf{c}_j^{\varphi_x} = \mathbf{c}_m$ (acting componentwise) for some $1\leqs m\leqs k-1$, and clearly $\mathbf{c}_k^{\varphi_x} = \mathbf{c}_k$ as $\mathbf{c}_k = (1,\dots,1)$ by our assumption above. Thus $\varphi_x$ induces a permutation $\pi\in S_k$, where
	\begin{equation*}
	j^\pi = m\mbox{ if }\mathbf{c}_j^{\varphi_x} = \mathbf{c}_m.
	\end{equation*}
	That is, $t_{i,j}^{\varphi_x} = t_{i,j^\pi}$ for each $i\in\{1,\dots,\ell-1\}$ and $j\in [k]$. Now
	\begin{equation*}
	\begin{aligned}
	D\mathbf{a}_i^{(\varphi_x,\dots,\varphi_x)\pi} &= D(\varphi_{t_{i,1}},\dots,\varphi_{t_{i,k}})^{(\varphi_x,\dots,\varphi_x)\pi} \\&= D(\varphi_{t_{i,1}^{\varphi_x}},\dots,\varphi_{t_{i,k}^{\varphi_x}})^\pi \\&= D(\varphi_{t_{i,1^\pi}},\dots,\varphi_{t_{i,k^\pi}})^\pi \\&= D(\varphi_{t_{i,1}},\dots,\varphi_{t_{i,k}}) = D\mathbf{a}_i.
	\end{aligned}
	\end{equation*}
	Therefore, $(\varphi_x,\dots,\varphi_x)\pi\in G_{(\Delta)}$, which concludes the proof.
\end{proof}

Now let us turn to the groups with $k = 2$.

\begin{prop}
	\label{p:diag_k=2_P=1}
	Suppose that $k = 2$ and $P = 1$. Then $G$ is semi-Frobenius.
\end{prop}

\begin{proof}
	By Theorem \ref{t:H_diag}, $b(G) = 3$. Let $x \in T \setminus \{1\}$. By \cite{GK_3/2}, there exists $y\in T$ such that $\la x,y\ra = T$. Let $\Delta := \{D,D(1,\varphi_x),D(1,\varphi_y)\}$ and suppose that $g\in G_{(\Delta)}$. Then $g = (\a,\a)$ for some $\a\in\Aut(T)$ as $P = 1$. It follows that $x^\a = x$ and $y^\a = y$, and since $\la x,y\ra = T$, we obtain $\a = 1$. Therefore, $\Delta$ is a base for $G$ and so $G$ is semi-Frobenius.
\end{proof}

\begin{lem}
	\label{l:diag_k=2_equiv_complete}
	Let $O\leqs\Out(T)$, $K := \Inn(T).O$, and $x, y \in T$. Suppose also that $G = T^2.(O\times S_2)$. Then $\{D,D(1,\varphi_x),D(1,\varphi_y)\}$ is a base for $G$ if and only if:
	\begin{itemize}\addtolength{\itemsep}{0.2\baselineskip}
		\item[{\rm (a)}] $C_{K}(x)\cap C_{K}(y) = 1$; and
		\item[{\rm (b)}] there is no $\a\in K$ such that $x^\a = x^{-1}$ and $y^\a = y^{-1}$.
	\end{itemize}
\end{lem}

\begin{proof}
    Let $\Delta:=\{D,D(1,\varphi_x),D(1,\varphi_y)\}$. If (a) fails to hold, then it is easy to see that $(\a,\a)\in G_{(\Delta)}$ for some $\a\in K \setminus \{1\}$, and so $\Delta$ is not a base for $G$. Similarly, if (b) fails to hold, then $(\a,\a)(1,2)\in G_{(\Delta)}$ for $\a\in K$ such that $x^\a = x^{-1}$ and $y^\a = y^{-1}$.
	
	Conversely, assume that both (a) and (b) hold, and suppose that $(\a,\a)\pi\in G_{(\Delta)}$, so that $\a\in K$. If $\pi = (1,2)$ then
	\begin{equation*}
	D(1,\varphi_x) = D(1,\varphi_x)^{(\a,\a)(1,2)} = D(1,\varphi_{x^\a})^{(1,2)} = D(\varphi_{x^\a},1) = D(1,\varphi_{(x^\a)^{-1}}),
	\end{equation*}
	so that $x^\a = x^{-1}$, and similarly $y^\a = y^{-1}$. However, this is incompatible with (b), and so $\pi = 1$. It follows that
	\begin{equation*}
	D(1,\varphi_x) = D(1,\varphi_x)^{(\a,\a)} = D(1,\varphi_{x^\a}),
	\end{equation*}
	and thus $\a\in C_{K}(x)$. Similarly, $\a\in C_K(y)$ and hence $\a = 1$ by condition (a). This shows that $\Delta$ is a base for $G$.
\end{proof}

As an immediate corollary, we deduce part of \cite[Lemma 3.5]{LMM_IBIS}, which is also recorded as \cite[Lemma 5.1]{H_diag}.

\begin{cor}
	\label{c:H_diag:l:5.1}
	Suppose that $G = T^2.(\Out(T)\times S_2)$, and let $x,y\in T$. Then the set $\{D,D(1,\varphi_x),D(1,\varphi_y)\}$ is a base for $G$ if and only if:
	\begin{itemize}\addtolength{\itemsep}{0.2\baselineskip}
		\item[{\rm (a)}] $C_{\Aut(T)}(x)\cap C_{\Aut(T)}(y) = 1$; and
		\item[{\rm (b)}] there is no $\a\in\Aut(T)$ such that $x^\a = x^{-1}$ and $y^\a = y^{-1}$.
	\end{itemize}
\end{cor}

\begin{lem}
	\label{l:diag_k=2_A5A6}
	Let $O\leqs \Out(T).$ Suppose that $P = S_2$ and $T\in\{A_5,A_6\}$. Then $G$ is not semi-Frobenius if and only if $G = T^2.(O\times S_2)$ with
	\begin{equation*}
	(T,\Inn(T).O) \in \{(A_5,A_5),(A_6,S_6),(A_6,\PGL_2(9))\}.
	\end{equation*}
\end{lem}

\begin{proof}
	This can be obtained with the aid of {\sc Magma}, noting that one can use Lemma \ref{l:diag_k=2_equiv_complete} to show that the three groups specified above are not semi-Frobenius.
\end{proof}

From now on we assume that $T\notin\{A_5,A_6\}$, so that $b(G) = 3$ by Theorem \ref{t:H_diag}. Then Corollary \ref{c:H_diag:l:5.1} implies that $T^2.(\Out(T)\times S_2)$ is semi-Frobenius if and only if for each non-identity element $x\in T$,
\begin{equation}
\label{e:star}
\tag{$\star$}
\mbox{there exists $y\in T$ such that Conditions (a) and (b) of Corollary \ref{c:H_diag:l:5.1} hold.}
\end{equation}

\begin{lem}
	\label{l:diag_star_prime}
	Suppose that $T\notin\{A_5,A_6\}$. Then $T^2.(\Out(T)\times S_2)$ is semi-Frobenius if and only if \eqref{e:star} holds for all $x \in T$ of prime order.
\end{lem}

\begin{proof}
	Let $m$ be an integer, and let $t \in T$ and $\a \in \Aut(T)$. Since $C_{\Aut(T)}(t)\leqs C_{\Aut(T)}(t^m)$ and $(t^m)^\a = (t^\a)^m$, the result follows.
\end{proof}

\begin{lem}
	\label{l:diag_star_centraliser_prime}
	Suppose that $T$ contains an element $y$ that is not $\Aut(T)$-conjugate to $y^{-1}$, with $|C_{\Aut(T)}(y)|$ prime. Then $T^2.(\Out(T)\times S_2)$ is semi-Frobenius.
\end{lem}

\begin{proof}
	Note that $T\notin\{A_5,A_6\}$. Additionally, $C_{\Aut(T)}(y) = \la y\ra$ is of prime order, and so $C_{\Aut(T)}(y)\cap C_{\Aut(T)}(x) = 1$ for all $x\in T\setminus\la y\ra$. Thus \eqref{e:star} holds for all $x \in T \setminus \{1\}$.
\end{proof}

\begin{prop}
	\label{p:diag_star_sporadic}
	Suppose that $k = 2$, and that $T$ is a sporadic simple group. Then $G$ is semi-Frobenius.
\end{prop}

\begin{proof}
	First note that $b(G) = 3$ by Theorem \ref{t:H_diag}, and so it suffices to consider the case where $G = T^2.(\Out(T)\times S_2)$. We divide the proof into four cases.
	
	\vs
	
	\noindent \emph{Case 1. $T\in\{\mathrm{M}_{12},\mathrm{M}_{22},\mathrm{HS},\mathrm{J}_1,\mathrm{J}_2,\mathrm{J}_3,\mathrm{McL},\mathrm{Suz},\mathrm{Ru},\mathrm{HN},\mathrm{He},\mathrm{Fi}_{22},\mathrm{Fi}_{24}',\mathrm{O'N}\}$.}
	
	\vs
	
	Computations in {\sc Magma} (involving sequential or random searches) show that for each conjugacy class representative $x \in T$ of prime order, there exists an element $y \in T$ such that $\langle x, y \rangle = T$ and $N_{\Aut(T)}(\la x \ra) \cap N_{\Aut(T)}(\la y \ra) = 1$. This first property implies that $C_{\Aut(T)}(x) \cap C_{\Aut(T)}(y) = C_{\Aut(T)}(T) = 1$, so that Condition (a) of Corollary \ref{c:H_diag:l:5.1} holds, while the second property implies Condition (b) of that corollary. Thus Lemma~\ref{l:diag_star_prime} yields the result.
	
	\vs
	
	\noindent \emph{Case 2. $T\in\{\mathrm{M}_{11},\mathrm{M}_{23},\mathrm{M}_{24},\mathrm{Co}_1,\mathrm{Co}_2,\mathrm{Co}_3,\mathrm{Fi}_{23},\mathrm{Th},\mathbb{B},\mathbb{M}\}$.}
	
	\vs
	
	We observe from the {\sc Atlas} \cite{ATLAS} that $\Aut(T) = T$ has an element $x$ that is not conjugate to $x^{-1}$, with $|C_T(x)|$ prime. Now apply Lemma \ref{l:diag_star_centraliser_prime}.
	
	\vs
	
	\noindent \emph{Case 3. $T \in \{\mathrm{Ly},\mathrm{J}_4\}$.}
	
	\vs

    The {\sc Atlas} again yields $T = \mathrm{Aut}(T)$, as well as the information that follows related to conjugacy and centraliser orders in $T$. Let $k:=33$ if $T = \mathrm{Ly}$, and $k:=35$ if $T = \mathrm{J}_4$. Then $T$ has a unique conjugacy class of cyclic subgroups of order $k$, and letting $a$ be a generator for such a subgroup, $C_T(a) = \langle a \rangle$, and $a$ and $a^{-1}$ lie in distinct $T$-conjugacy classes. Using the (primarily) computational methods described in the following two paragraphs, we show that for each $x \in T$ of prime order, there exists a $T$-conjugate $a'$ of $a$ such that the commutator $[a',x]$ is not equal to $1$, so that $1 = C_{\la a' \rangle}(x) = C_T(x) \cap C_T(a')$. Thus \eqref{e:star} holds for $x$ with $y = a'$, and the result follows from Lemma~\ref{l:diag_star_prime}.

    If $T = \mathrm{Ly}$, then we construct $T$ as a permutation group in {\sf GAP} \cite{GAP} using the \texttt{AtlasGroup} function from the AtlasRep \cite{AtlasRep} package, and subsequently verify that the required commutator property holds (up to conjugacy of elements).
    
    If instead $T = \mathrm{J}_4$, then we construct $T$ as a subgroup of $\GL_{112}(2)$, using {\sc Magma}'s database of {\sc Atlas} groups. If $|x| \notin \{2,11\}$, then $T$ has a unique conjugacy class of cyclic subgroups of order $|x|$, and so we can obtain a generator for a conjugate of $\langle x \rangle$ by generating a random element of order $|x|$. If instead $|x| = 2$, then either $x$ is conjugate to $y^4$ for each $y \in T$ of order $8$, or $x$ is conjugate to $y^{14}$ for each $y \in T$ of order $28$. Thus we can construct a conjugate of $x$ by randomly generating an appropriate $y$. In all of these cases, we verify the specified commutator property computationally. Finally, if $|x| = 11$, then $|a|$ and $|C_T(x)|$ are coprime, and hence $[a,x] \ne 1$.
\end{proof}

\begin{prop}
	\label{p:diag_star_alternating}
	Suppose that $k = 2$ and $T = A_n$ for some $n\geqs 7$. Then $G$ is semi-Frobenius.
\end{prop}

\begin{proof}
	The cases where $n\in\{7,8,9\}$ can be handled using {\sc Magma}, so we will suppose that $n\geqs 10$. Note that $b(G) = 3$ by Theorem \ref{t:H_diag}, and so we may assume that $G = T^2.(\Out(T)\times S_2)$. We will appeal to Lemma~\ref{l:diag_star_prime} and prove the result by showing that \eqref{e:star} holds for all elements $x \in T$ of prime order. Conjugating by an element of $S_n$ if necessary, we may assume that
	\begin{equation}\label{e:cycle}
	x = (1,2,\dots,p)(p+1,p+2,\dots,2p)\cdots((s-1)p+1,(s-1)p+2,\dots,sp)
	\end{equation}
	for some prime $p$ and some positive integer $s$, so that the integers from $1$ to $sp$ appear in consecutive order.
	
	Let $m := n$ if $n$ is odd and $m := n-1$ if $n$ is even. Suppose first that $x = (1,2)(3,4)$. We claim that \eqref{e:star} holds with $y = (1,4,3,5,2,6,7,\dots,m)$ (so that all integers from $6$ to $m$ appear in consecutive order). To see this, first note that $C_{S_n}(y) = \la y\ra$, and that the elements of $S_n$ inverting $y$ by conjugation are precisely the elements of the right coset $\langle y \rangle \sigma$, where
	\[\sigma:=(4,m)(3,m-1)(5,m-2)(2,m-3)(6,m-4)(7,m-5)\cdots\left(\frac{m+1}{2},\frac{m+3}{2}\right).\]
	It is easy to check that if $g\in S_n\setminus\{1\}$ satisfies $y^g\in\{y,y^{-1}\}$ and $1^g\leqs 4$, then $2^g > 4$ or $4^g > 4$, and thus $x^g\ne x = x^{-1}$.
	
	To complete the proof, suppose that $x \ne (1,2)(3,4)$. We claim that \eqref{e:star} holds with $y = (1,3,4,2,5,6,\dots,m)$. By the definition of $x$ in \eqref{e:cycle}, $x$ cannot lie in $\langle y \rangle$. In what follows, we write ordered pairs using square brackets to avoid confusion with permutations.
	
	First, we will show that $C_{S_n}(x)\cap C_{S_n}(y) = 1$. Note that $C_{S_n}(y) = \la y\ra$, so for each non-identity $g\in C_{S_n}(y)$, either
	\begin{equation*}
	[1^g,2^g] \in \{[2,7],[3,5],[4,6],[m-2,1],[m-1,3],[m,4]\},
	\end{equation*}
	or $1^g \in \{5,\ldots,m-3\}$ and $2^g = 1^g + 3$. Additionally, $3^g = m-1$ if $1^g = m-2$. On the other hand, we see from~\eqref{e:cycle} that each $h\in C_{S_n}(x)$ maps $1$ and $2$ to points in a common $p$-cycle of $x$, and $2^h = 1^h + 1$, interpreted within that cycle. Moreover, if $[1^h,2^h] = [m-2,1]$ then $(1,\ldots,m-2)$ is a $p$-cycle of $x$ and $3^h = 2$. It follows that $C_{S_n}(x)\cap C_{S_n}(y) = 1$.
	
	Finally, we will show that no element of $S_n$ inverts both $x$ and $y$ via conjugation. To see this, first observe that the elements of $S_n$ inverting $y$ are precisely the elements of $\langle y \rangle\tau$, where
	\[ \tau :=	(3,m)(4,m-1)(2,m-2)(5,m-3)(6,m-4)\cdots\left(\frac{m+1}{2},\frac{m+3}{2}\right).\]
	Assume that $g$ lies in this coset. Then either
	\begin{equation*}
	[1^g,2^g]\in\{[1,m-2],[2,1],[3,m-1],[4,m],[5,3],[6,4],[7,2]\},
	\end{equation*}
	or $1^g \in \{8,\ldots,m\}$ and $2^g = 1^g-3$. Furthermore, if $g$ fixes $1$, then $3^g = m$, and if $[1^g,2^g] = [2,1]$, then $[3^g,5^g,6^g] = [4,m,m-1]$. Now let $h$ be an element of $S_n$ that inverts $x$. Then $h$ maps $1$ and $2$ to points in a common $p$-cycle of $x$, and $2^h = 1^h-1$, interpreted within that cycle. In addition, $3^h = m-3$ if $[1^h,2^h] = [1,m-2]$, and if $[1^h,2^h] = [2,1]$, then either $p = 2$ or $3^h = p$. Thus, if $h$ also inverts $y$, then $[1^h,2^h,3^h,5^h,6^h] = [2,1,4,m,m-1]$, and the fact that $p$ is prime yields $p = 2$. Our assumption that $x \ne (1,2)(3,4)$ now gives $sp \ge 6$, and $[5^h,6^h] = [m,m-1]$ implies that $(m-1,m)$ is a cycle of $x$, so that $m$ must be even. However, $m$ is odd, a contradiction. Hence the result follows.
 \end{proof}
 
Theorem \ref{thm:diag_semi} follows from Propositions \ref{p:diag_l+2_complete}, \ref{p:diag_l+1_incomplete}, \ref{p:diag_k=2_P=1}, \ref{p:diag_star_sporadic} and \ref{p:diag_star_alternating}. We also present an infinite family of non-semi-Frobenius diagonal type primitive groups with $k=2$.

\begin{prop}
    \label{p:diag_star_psl2}
    Suppose that $G = T^2.(\Out(T)\times S_2)$, where $T = \LL_2(q)$ for some $q\geqs 11$. Then $G$ is not semi-Frobenius.
\end{prop}

\begin{proof}
    Let $x\in T$ be an element of order $(q+1)/(2,q-1)$. We claim that \eqref{e:star} does not hold, which implies the desired result. Letting $R:=\PGL_2(q)$, we note that $C_T(x) = \la x\ra$, $C_{R}(x) \cong C_{q+1}$, and $N_{R}(\la x\ra)\cong D_{2(q+1)}$ is a maximal subgroup of $R$. In particular, each element of $N_{R}(\la x\ra)$ either centralises or inverts $x$.
    
    Write $q = p^f$ with $p$ prime, and let $y\in T$. Then either $|y| = p$, or $y$ lies in a cyclic subgroup of $T$ of order $(q-1)/(2,q-1)$ or $(q+1)/(2,q-1)$. 
    We shall divide the rest of the proof into four cases, which together account for all possible choices for $y$; in each case, we will show that $y$ does not satisfy \eqref{e:star}.

    \vs
	
	\noindent \emph{Case 1. $|y| = 2$.}

    \vs

    The normaliser $N_T(\la x\ra)$ is the unique maximal subgroup of $T$ containing $x$. Thus if $y\notin N_T(\la x\ra)$, then $\la x,y\ra = T$. It follows (see \cite[Theorem 3]{Singerman} and \cite[p.~582]{LL_gen}) that there exists $\a\in\Aut(T)$ such that $(x,y)^\a = (x^{-1},y^{-1})$. If instead $y\in N_T(\la x\ra) \setminus \la x\ra$, then $(x,y)^y = (x^{-1},y^{-1})$, and if $y\in \la x\ra$, then $y\in C_{\Aut(T)}(x)\cap C_{\Aut(T)}(y)$. Hence $y$ does not satisfy \eqref{e:star}.

    \vs
	
	\noindent \emph{Case 2. $p$ is odd and $|y| = p$.}

    \vs
    
    Let $H$ be the unique maximal subgroup of $R$ of type $P_1$ that contains $y$. Then $H\cong C_p^f \sd C_{q-1}$, and each involution of $H$ inverts $y$. Since $(q-1,q+1)= 2$ and any subgroup of $H$ of order $4$ is cyclic, we see that $J:=H\cap N_{R}(\la x\ra)$ 
    has order at most $2$. In fact, $|H|\cdot |N_{R}(\la x\ra)| > |R|$, which implies that $|J| = 2$. Let $\a$ be the unique involution of $J$. 
    Now, $\a$ is contained in a cyclic subgroup of $H$ of order $q-1$. Hence if $\a$ centralises $x$, then $\a$ lies in a proper subgroup of $R$ of order divisible by  both $q-1$ and $q+1$. However, no such subgroup exists, and so $(x,y)^\a = (x^{-1},y^{-1})$.

    \vs

    \noindent \emph{Case 3. $|y| \ne  2$ and $|y|$ divides $(q+1)/(2,q-1)$.}

    \vs
    
    Here, $y$ is contained in a cyclic subgroup of $T$ of order $(q+1)/(2,q-1)$. Since all such subgroups of $T$ are conjugate, we obtain that $y\in\la x^g\ra$ for some $g\in T$. It is also clear that the maximal subgroup $N_{R}(\la x^g\ra) \cong D_{2(q+1)}$ of $R$ is equal to $N_{R}(\la y\ra)$. If $y\in\la x\ra$, then clearly $C_{\Aut(T)}(x)\cap C_{\Aut(T)}(y)\ne 1$, so we shall assume that $y\notin\la x\ra$. By inspecting \cite[Table 2]{FI_subdeg}, we see that $N:=N_{R}(\la x \ra) \cap N_{R}(\la y\ra) \ne 1$. As $\la x \ra \cap \la x^g \ra = 1$, it follows that there exists $\a\in N$ such that $(x,y)^\a = (x^{-1},y^{-1})$.

     \vs

    \noindent \emph{Case 4. $|y|  \ne 2$ and $|y|$ divides $(q-1)/(2,q-1)$.}

    \vs
    
    We shall identify elements of $R$ with corresponding matrices in $\GL_2(q)$. In this case, $y$ lies in a cyclic subgroup of $T$ of order $(q-1)/(2,q-1)$, and so (conjugating $x$ and $y$ by a common element of $T$ if necessary) we may assume without loss of generality that $y$ is a diagonal matrix. Suppose that
    \begin{equation*}
        x = \begin{pmatrix}
            a&b\\
            c&d
        \end{pmatrix} \mbox{ for $a,b,c,d \in \mathbb{F}_q$,}
    \end{equation*}
    noting that $b,c \neq 0$ since $|x| = (q+1)/(2,q-1)$. Then the matrix
    \begin{equation*}
        \begin{pmatrix}
            0&1\\
            -cb^{-1}&0
        \end{pmatrix}
    \end{equation*}
    inverts each of $x$ and $y$, and so \eqref{e:star} does not hold.
\end{proof}

\begin{rem}
    Suppose that $G = T^2.(\Out(T)\times S_2)$ with $T\notin\{A_5,A_6\}$ (so that $b(G) = 3$), and let \[N := \{t\in T: D(1,\varphi_t) \text{ is a neighbour of } D \text{ in } \Sigma(G)\}.\]
    Then $\Sigma(G)$ has diameter at most $2$ (equivalently, $G$ satisfies Conjecture~\ref{conj:BG}) if and only if $N^2=T$ (that is, any $t\in T$ can be written as $t = x_1x_2$ for some $x_i\in N$). 
    To observe this, note that for $g = (1,x^{-1})\in T^2\leqs G$, \[(D(1,\varphi_x), D(1,\varphi_y))^g = (D,D(1,\varphi_{yx^{-1}})).\]
    In particular, $(D(1,\varphi_x), D(1,\varphi_y))$ is an arc in $\Sigma(G)$ if and only if $(D,D(1,\varphi_{yx^{-1}}))$ is. Hence $(D,D(1,\varphi_{x_2}),D(1,\varphi_{x_1x_2}))$ is a $2$-arc from $D$ to $D(1,\varphi_t),$ so $\mathrm{diam}(\Sigma(G)) \leq 2$. The converse also holds by a similar argument. 
    As an example, in the case of $T\cong \LL_2(q)$ with $q = 7$ or $q\geqs 11$, the elements of $T$ of order $(q-1)/(2,q-1)$ lie in $N$ (as noted in the proof of \cite[Proposition 5.3]{H_diag}). By \cite[pp.~240--241]{arad}, for the conjugacy class $C$ of an element of order $(q-1)/(2,q-1)$, we have $C^2=T,$ and so it follows that $N^2 = T$ and $\mathrm{diam}(\Sigma(G)) \leq 2.$
\end{rem}

\subsection{Arc-transitivity}

\label{ss:diag_arc}

In this subsection, we will prove Theorem \ref{thm:diag_arc-trans}, which classifies the diagonal type primitive groups $G$ such that $\Sigma(G)$ is $G$-arc-transitive. We first record \cite[Theorem 2]{H_diag}, which deals with the case $b(G) = 2$, noting that in this case $\reg(G) = 1$ if and only if $\Sigma(G)$ is $G$-arc-transitive.

\begin{thm}{\cite[Theorem 2]{H_diag}}
	\label{t:H_diag_r=1}
	Suppose that $b(G) = 2$. Then $\Sigma(G)$ is $G$-arc-transitive if and only if $T=A_5$, $k\in\{3,57\}$ and $G = T^k.(\Out(T)\times S_k)$.
\end{thm}

Recall that for an element $\mathbf{x} = (\varphi_{t_1},\dots,\varphi_{t_k})\in\Inn(T)^k$, the partition $\mathcal{P}^{\mathbf{x}} = \{\mathcal{P}^{\mathbf{x}}_t:t\in T\}$ of $[k]$ is defined so that $j\in \mathcal{P}_t^{\mathbf{x}}$ if $t_j = t$. 

\begin{lem}
	\label{l:diag_subaction}
	Let $\mathbf{x},\mathbf{y} \in \Inn(T)^k$ be such that $D\mathbf{x}$ and $D\mathbf{y}$ lie in a common $G_D$-orbit. Then there exist $g\in T$ and $\a\in\Aut(T)$ such that for each integer $m$,
	\begin{equation*}
	g\{t:|\mathcal{P}_t^{\mathbf{x}}| = m\}^\a = \{t:|\mathcal{P}_t^{\mathbf{y}}| = m\}.
	\end{equation*}
\end{lem}

\begin{proof}
	Write $\mathbf{x} = (\varphi_{x_1},\dots,\varphi_{x_k})$ and $\mathbf{y} = (\varphi_{y_1},\dots,\varphi_{y_k})$, and let $\a\in\Aut(T)$ and $\pi\in S_k$ be such that $D\mathbf{x}^{(\a,\dots,\a)\pi} = D\mathbf{y}$. Then
	\begin{equation*}
	D(\varphi_{x_1^\a},\dots,\varphi_{x_k^\a}) = D(\varphi_{y_{1^\pi}},\dots,\varphi_{y_{k^\pi}}).
	\end{equation*}
	It follows that there exists $g\in T$ such that $gx_j^\a = y_{j^\pi}$ for all $j\in [k]$, 
    and that $(\mathcal{P}_{x_j}^{\mathbf{x}})^\pi = \mathcal{P}_{y_{j^\pi}}^{\mathbf{y}}$. In particular, $|\mathcal{P}_{x_j}^{\mathbf{x}}| = |\mathcal{P}_{y_{j^\pi}}^{\mathbf{y}}| = |\mathcal{P}_{gx_j^\a}^{\mathbf{y}}|$, and thus
	\begin{equation*}
	\{gt^\a:|\mathcal{P}_t^{\mathbf{x}}| = m\} = \{t:|\mathcal{P}_t^{\mathbf{y}}| = m\}
	\end{equation*}
	as desired.
\end{proof}

We first consider the case where $k = 2$. The following result is a stronger version of \cite[Corollary 5.4]{H_diag}. Here, we write $\LL_n^+(q):=\LL_n(q)$ and $\LL_n^-(q):=\UU_n(q)$.

\begin{lem}
	\label{l:diag_arc-trans_star}
	Let $O\leqs\Out(T)$ and $K := \Inn(T).O$, with $O\ne \Out(T)$ if $T\in\{A_5,A_6\}$. Then there exists a pair $(x,y)$ of elements of $T$ such that:
	\begin{itemize}\addtolength{\itemsep}{0.2\baselineskip}
		\item[{\rm (a)}] $|x|\ne |y|$;
		\item[{\rm (b)}] $C_{K}(x)\cap C_{K}(y) = 1$; and
		\item[{\rm (c)}] there is no $\a\in K$ such that $x^\a = x^{-1}$ and $y^\a = y^{-1}$.
	\end{itemize}
\end{lem}

\begin{proof}
	First assume that $T$ is not isomorphic to $\LL_2(q)$ or $\LL_3^\e(q)$ for $\e \in \{+,-\}$. One can handle the case $T = A_7$ with the aid of {\sc Magma}. If $T\ne A_7$, then \cite[Theorem 1.1]{LL_gen} implies that there exists a generating pair $(x,y)$ for $T$ with $|x| = 2$, such that there is no $\a\in\Aut(T)$ with $(x,y)^\a = (x^{-1},y^{-1})$. As $T$ is not dihedral, $|y|\ne |x|$.
	
	Next, assume that $T = \LL_2(q)$ for $q\geqs 5$, with $q = p^f$ for some prime $p$. If $q\in\{5,9\}$, then $T\in\{A_5,A_6\}$, and one can show that the result holds using {\sc Magma} (noting that $O < \Out(T)$ by our assumption). Now assume that $q\notin\{5,9\}$, and let $\lambda$ be a primitive element of $\mathbb{F}_q^\times$. 
    Additionally, let $x\in T$ be the image of
	\begin{equation*}
	\widehat{x} = 
	\begin{pmatrix}
	\lambda&0\\
	0&\lambda^{-1}
	\end{pmatrix} \in \SL_2(q),
	\end{equation*}
	and let $y\in T$ be the image of
	\begin{equation*}
	\widehat{y} = 
	\begin{pmatrix}
	1&\mu\\
	0&1
	\end{pmatrix} \in \SL_2(q)
	\end{equation*}
	for some $\mu\in\mathbb{F}_q^\times$. Then $|x| =(q-1)/(2,q-1)$ and $|y| = p$. We claim that conditions (b) and (c) hold. To see this, first note that $C_{\mathrm{\Gamma L}_2(q)}(\widehat{x}) = C_{\mathrm{\GL}_2(q)}(\widehat{x})$, and thus
	\begin{equation*}
	C_{\PGammaL_2(q)}(x) = C_{\PGL_2(q)}(x)\cong C_{q-1}.
	\end{equation*}
	Moreover, $C_{\PGL_2(q)}(y)\cong C_p^f$, and so
	\begin{equation*}
	C_{\Aut(T)}(x)\cap C_{\Aut(T)}(y) = C_{\PGL_2(q)}(x)\cap C_{\PGL_2(q)}(y) = 1,
	\end{equation*}
	which gives (b). Observe next that $\widehat{x}^g = \widehat{x}^{-1}$, where
	\begin{equation*}
	g = 
	\begin{pmatrix}
	0&1\\
	1&0
	\end{pmatrix}\in\GL_2(q).
	\end{equation*}
	It follows that if $h\in\mathrm{\Gamma L}_2(q)$ and $\widehat{x}^h = \widehat{x}^{-1}$, then $h$ lies in the coset $C_{\mathrm{\Gamma L}_2(q)}(\widehat{x})g = C_{\mathrm{\GL}_2(q)}(\widehat{x})g$. Thus
	\begin{equation*}
	h = 
	\begin{pmatrix}
	a&0\\
	0&b
	\end{pmatrix}
	\begin{pmatrix}
	0&1\\
	1&0
	\end{pmatrix}
	= 
	\begin{pmatrix}
	0&a\\
	b&0
	\end{pmatrix}
	\end{equation*}
	for some $a,b\in\mathbb{F}_q^\times$. However,
	\begin{equation*}
	\widehat{y}^h = 
	\begin{pmatrix}
	1&0\\
	a^{-1}b\mu&1
	\end{pmatrix} \ne
	\begin{pmatrix}
	1&-\mu\\
	0&1
	\end{pmatrix}
	= \widehat{y}^{-1}.
	\end{equation*}
  Condition (c) now follows, and the proof is complete for $T = \LL_2(q)$.
	
	It remains to consider the case $T = \LL_3^\e(q)$, with $\e \in \{+,-\}$ and $q \ge 3$. It suffices to find $x,y\in T$ of distinct orders such that $N_{\Aut(T)}(\la x\ra)\cap N_{\Aut(T)}(\la y\ra) = 1$. For $q \le 32$, it is routine to find these elements by random search with the aid of {\sc Magma}. Thus, we may assume that $q > 32$. Let $H_1$ be a maximal subgroup of $\Aut(T)$ of type $\GL_1^\e(q^3)$, so that $H_1 = N_{\Aut(T)}(\la x\ra)$ for some $x\in T$ with $|x| = (q^2+\e q+1)/(3,q-\e)$. In addition, let $y$ be an element of $T$ with a preimage
	\begin{equation*}
	\widehat{y} = 
	\begin{pmatrix}
	A&\\
	&\zeta
	\end{pmatrix}\in\SL_3^\e(q),
	\end{equation*}
	for some $A\in\GL_2(q)$ of order $q^2-1$, with $\zeta$ chosen so that $\det(\widehat{y}) = 1$. Then $|\widehat{y}| = q^2-1$ and $|y| = (q^2-1)/(3,q-\e)$. Let $H_2 := N_{\Aut(T)}(\la y\ra)$. By \cite[Satz II.7.3(a)]{Hu_book}, we have $H_2\cap T = \la y\ra.2$. It suffices to show that $H_1^g\cap H_2 = 1$ for some $g\in \Aut(T)$, or equivalently, that $\Aut(T)$ has a regular orbit on $\Gamma_1\times \Gamma_2$, where $\Gamma_i := [\Aut(T):H_i]$. Recall that
	\begin{equation*}
	\fpr(z,\Gamma_i) = \frac{|z^{\Aut(T)}\cap H_i|}{|z^{\Aut(T)}|}
	\end{equation*}
	is the fixed point ratio of $z\in \Aut(T)$ on $\Gamma_i$. Letting $R(\Aut(T))$ be a set of representatives of the $\Aut(T)$-conjugacy classes of elements of prime order in $H_1$, arguing as in Section \ref{ss:pre_prob} shows that $\Aut(T)$ has a regular orbit on $\Gamma_1\times \Gamma_2$ if
	\begin{equation}
	\label{e:reg_prod}
	\begin{aligned}
	m:&=\sum_{z\in R(\Aut(T))}|z^{\Aut(T)}|\cdot \fpr(z,\Gamma_1)\cdot\fpr(z,\Gamma_2) \\
	& = \sum_{z\in R(\Aut(T))}\frac{|z^{\Aut(T)}\cap H_1|\cdot |z^{\Aut(T)}\cap H_2|}{|z^{\Aut(T)}|} < 1.
	\end{aligned}
	\end{equation}
	
	Adapting the method used to prove \cite[Lemma 6.4]{B_sol}, let $z\in H_1$ be an element of prime order $r$. First assume that $z$ is unipotent or semisimple, and we apply the bound given in the proof of \cite[Lemma 6.4]{B_sol}, which is
	\begin{equation*}
	|z^{\Aut(T)}|\geqs  \frac{1}{3}q^3(q-1)(q^2-q+1)=:c_1.
	\end{equation*}
	We note that $|H_1\cap\PGL_3^\e(q)|\leqs 3(q^2+q+1) =: a_1$ and $|H_2\cap \PGL_3^\e(q)| \leqs 6(q^2-1) =: b_1$.
 
 Now assume that $z$ is a field automorphism with $r$ odd. Here $r\geqs 5$ as noted in the proof of \cite[Lemma 6.4]{B_sol}, so we have
	\begin{equation*}
	|z^{\Aut(T)}|\geqs |z^{\PGL_3^\e(q)}| = \frac{|\PGL_3^\e(q)|}{|\PGL_3^\e(q^{1/r})|} > \frac{1}{2}q^{32/5}=:c_2,
	\end{equation*}
	and there are at most $3(q^2+q+1)\log_2 q =: a_2$ and $6(q^2-1)\log_2 q =: b_2$ of these elements in $H_1$ and $H_2$, respectively.
 
 Assume next that $z$ is an involutory graph automorphism. Then $|C_{\PGL_3^\e(q)}(z)| = |\Sp_2(q)|$ and so $|z^{\Aut(T)}|\geqs q^2(q^3-1) =: c_3$. Moreover, $z$ inverts the normal subgroup $C_{(q^2+\e q+1)/(3,q-\e)}$ of $H_1$ and the normal subgroup $C_{(q^2-1)/(3,q-\e)}$ of $H_2$. This implies that both $C_{(q^2+\e q+1)/(3,q-\e)}.\la z\ra$ and $C_{(q^2-1)/(3,q-\e)}.\la z\ra$ are dihedral, so $H_1$ contains at most $a_3:=q^2+q+1$ involutory graph automorphisms, and $H_2$ contains at most $b_3 := q^2-1$ involutory graph automorphisms.
 
 It remains to consider the case where $\e = +$ and $z$ is an involutory field or graph-field automorphism. Here,
	\begin{equation*}
	|z^{\Aut(T)}|\geqs |z^{\PGL_3(q)}| \geqs \frac{|\PGL_3(q)|}{|\PGU_3(q^{1/2})|} = q^{3/2}(q+1)(q^{3/2}-1) =: c_4,
	\end{equation*}
	and (as noted in the proof of \cite[Lemma 6.4]{B_sol}) there are fewer than $a_4:=2(q+q^{1/2}+1)$ of these elements in $H_1$. We also set $b_4:=12(q^2-1)\geqs |(H_2\cap \PGL_3^\e(q)).\la z\ra|$.
	
	Finally, we conclude that the number $m$ defined in \eqref{e:reg_prod} is less than $\sum_{i = 1}^4 a_ib_i/c_i$, which is less than $1$ if $q > 32$. As noted above, this implies that $N_{\Aut(T)}(\la x^g\ra) \cap N_{\Aut(T)}(\la y\ra) = 1$ for some $g\in G$, which completes the proof.
\end{proof}

\begin{prop}
	\label{p:diag_arc-trans_k=2}
	Suppose that $k = 2$. Then $\Sigma(G)$ is not $G$-arc-transitive, and hence $\reg(G)\geqs2$.
\end{prop}

\begin{proof}
	If $P = 1$, or if $T \in \{A_5,A_6\}$ and $G = T^2.(\Out(T) \times S_2)$, then $G$ is semi-Frobenius by Proposition~\ref{p:diag_k=2_P=1} and Lemma~\ref{l:diag_k=2_A5A6}. However, $G$ is not $2$-transitive, and so $\Sigma(G)$ is not $G$-arc-transitive. In each other case, Theorem~\ref{t:H_diag} gives $b(G) = 3$, and all hypotheses in the statement of Lemma \ref{l:diag_arc-trans_star} hold. Let $x$ and $y$ be the elements described in that lemma. Then Corollary~\ref{c:H_diag:l:5.1} shows that $\{D,D(1,\varphi_{x}),D(1,\varphi_{y})\}$ is a base for $G$, and so $D(1,\varphi_x)$ and $D(1,\varphi_y)$ are adjacent to $D$ in $\Sigma(G)$. On the other hand, since $|x| \notin \{1,|y|\}$, Lemma \ref{l:diag_subaction} implies that $D(1,\varphi_x)$ and $D(1,\varphi_y)$ lie in distinct $G_D$-orbits. Therefore,  $\Sigma(G)$ is not $G$-arc-transitive.
\end{proof}

Recall that the \emph{holomorph} $\Hol(T) = T \sd \Aut(T)$ of $T$ has a faithful, primitive action on $T$. In fact, $\Hol(T) = T^2.\Out(T)$ is a diagonal type group with trivial top group. Each element of $\Hol(T)$ can be written uniquely as $\widehat{g}\a$, where $\widehat{g}\in T$ acts on $T$ by left translation and $\a\in\Aut(T)$ acts naturally on $T$. More precisely,
\begin{equation*}
t^{\widehat{g}\a} = (g^{-1}t)^\a = (g^{-1})^\a t^\a
\end{equation*}
for every $t\in T$. For a subset $S\subseteq T$, we write $\mathrm{Hol}(T,S)$ for the setwise stabiliser of $S$ in $\mathrm{Hol}(T)$, noting that $\Hol(T,S) = 1$ if and only if $S$ lies in a regular $\Hol(T)$-orbit.
\begin{lem}
	\label{l:H_diag_l:5.5}
	Suppose that $P\in\{A_k,S_k\}$, and that $|T|^{\ell-1}<k\leqs |T|^{\ell}-3$ for some $\ell\geqs 2$. Then there exists $\mathbf{x}\in\Inn(T)^k$ such that the partition $\mathcal{P}^{\mathbf{x}} = \{\mathcal{P}_t^{\mathbf{x}}:t\in T\}$ of $[k]$ satisfies the following properties.
	\begin{itemize}\addtolength{\itemsep}{0.2\baselineskip}
		\item[{\rm (P1)}] $|\mathcal{P}_t^{\mathbf{x}}|\leqs |T|^{\ell-1}$ for all $t\in T$.
		\item[{\rm (P2)}] $|\mathcal{P}_1^{\mathbf{x}}|\ne 0$ and $\mathrm{Hol}(T,S) = 1$, where
		\begin{equation*}
		S = \{t\in T :|\mathcal{P}_t^{\mathbf{x}}| = |\mathcal{P}_1^{\mathbf{x}}|\}.
		\end{equation*}
		\item[{\rm (P3)}] There exists $t\in T \setminus \{1\}$ such that $|\mathcal{P}_t^{\mathbf{x}}|\in\{1,|T|^{\ell-1}-1\}$.
		\item[{\rm (P4)}] If $k \ne |T|^\ell-3$, then there exist $t_1,t_2\in T\setminus S$ with $|\mathcal{P}_{t_1}^{\mathbf{x}}|\ne |\mathcal{P}_{t_2}^{\mathbf{x}}|$.
        \item[{\rm (P5)}] If $k = |T|^\ell-3$, then $|S| = |T| - 3$, $|\mathcal{P}_t^{\mathbf{x}}| = |T|^{\ell-1}$ for all $t \in S$, and $|\mathcal{P}_t^{\mathbf{x}}| = |T|^{\ell-1}-1$ for all $t \in T \setminus S$.
	\end{itemize}
 Moreover, $D$ is adjacent in $\Sigma(G)$ to $D\mathbf{x}$ for every such $\mathbf{x}$. 
\end{lem}

\begin{proof}
	The existence of an element $\mathbf{x}\in\Inn(T)^k$ satisfying conditions (P1)--(P3) follows from \cite[Lemma 5.5]{H_diag}, and its proof shows that (P4) and (P5) also hold. As noted in the proof of \cite[Proposition 5.8]{H_diag}, $\{D,D\mathbf{x}\}$ can be extended to a base for $G$ of size $\ell+1$. The result therefore follows from Theorem~\ref{t:H_diag}, which yields $b(G) = \ell+1$.
\end{proof}

\begin{prop}
	\label{p:diag_arc-trans_l+1}
	Suppose that $P\in\{A_k,S_k\}$, and that $|T|^{\ell-1} < k \leqs |T|^\ell-3$ for some $\ell\geqs 2$. Then $\Sigma(G)$ is not $G$-arc-transitive, and hence $\reg(G)\geqs 2$.
\end{prop}

\begin{proof}
	Let $\mathbf{x}\in\Inn(T)^k$ be as in Lemma \ref{l:H_diag_l:5.5}, so that $D$ and $D\mathbf{x}$ are adjacent in $\Sigma(G)$. First assume that $k\ne |T|^\ell-3$. Additionally, let $t_1,t_2 \in T$ be the elements described in (P4), and define $\mathbf{y}\in\Inn(T)^k$ by setting
	\begin{equation*}
	\mbox{$\mathcal{P}_{t_1}^{\mathbf{y}} = \mathcal{P}_{t_2}^{\mathbf{x}}$, $\mathcal{P}_{t_2}^{\mathbf{y}} = \mathcal{P}_{t_1}^{\mathbf{x}}$, and $\mathcal{P}_{t}^{\mathbf{y}} = \mathcal{P}_{t}^{\mathbf{x}}$ for $t\in T\setminus\{t_1,t_2\}$.}
	\end{equation*}
	By applying Lemma \ref{l:H_diag_l:5.5} once again, $D$ and $D\mathbf{y}$ are adjacent in $\Sigma(G)$. Suppose that $D\mathbf{x}$ and $D\mathbf{y}$ lie in a common $G_D$-orbit. Then letting $S$ be as in (P2), Lemma \ref{l:diag_subaction} shows that there exist $g \in T$ and $\a \in \Aut(T)$ such that $gS^\a = S$. Since $\Hol(T,S) = 1$ by (P2), we obtain $g = \a = 1$. Lemma \ref{l:diag_subaction} now implies that $|\mathcal{P}_t^{\mathbf{x}}| = |\mathcal{P}_t^{\mathbf{y}}|$ for all $t\in T$. However, this contradicts (P4), which yields $|\mathcal{P}_{t_1}^{\mathbf{x}}| \ne |\mathcal{P}_{t_2}^{\mathbf{x}}| = |\mathcal{P}_{t_1}^{\mathbf{y}}|$.
	
	To complete the proof, we assume that $k = |T|^\ell-3$. Write $T^{\ell-1} = \{\mathbf{b}_1,\dots,\mathbf{b}_{|T|^{\ell-1}}\}$ and $\mathbf{b}_h = (a_{1,h},\dots,a_{\ell-1,h})$ for each $h \in \{1,\ldots,|T|^{\ell-1}\}$, with $\mathbf{b}_{|T|^{\ell-1}} := (1,\dots,1)$. In addition, define $\mathbf{a}_i := (\varphi_{t_{i,1}},\dots,\varphi_{t_{i,k}})\in\Inn(T)^k$ for $i \in \{1,\ldots,\ell-1\}$ by setting $t_{i,j} = a_{i,h}$ where $j$ is the $h$-th smallest integer in $\mathcal{P}_r^{\mathbf{x}}$ for the appropriate $r \in T$, noting that this is possible by (P5). Then it follows from (P5) (as observed in the proof of \cite[Proposition 5.8]{H_diag}) that the set $\{D,D\mathbf{x},D\mathbf{a}_1,\dots,\mathbf{a}_{\ell-1}\}$ is a base for $G$ of size $b(G) = \ell+1$. In particular, $D$ and $D\mathbf{a}_i$ are adjacent in $\Sigma(G)$. Since $|\mathcal{P}_t^{\mathbf{x}}|\in\{|T|^{\ell-1},|T|^{\ell-1}-1\}$ for all $t\in T$ by (P5), and since $\mathbf{b}_{|T|^{\ell-1}} := (1,\dots,1)$, we obtain $|\mathcal{P}_1^{\mathbf{a}_i}| = |T|^{\ell-1}-3$. Hence by Lemma \ref{l:diag_subaction}, $(D,D\mathbf{x})$ and $(D,D\mathbf{a}_i)$ are arcs of $\Sigma(G)$ lying in distinct $G$-orbits, and the proof is complete.
\end{proof}

\begin{lem}
	\label{l:diag_arc-trans_notSk}
	Suppose that $P\in\{A_k,S_k\}$, that $k\in\{|T|^\ell-2,|T|^\ell-1,|T|^\ell\}$ for some $\ell\geqs 2$, and that $G$ does not contain $S_k$. Then $\Sigma(G)$ is not $G$-arc-transitive, and hence $\reg(G)\geqs 2$.
\end{lem}

\begin{proof}
	First note by \cite[Theorem 1.2]{BGK_spread} that $T$ has a conjugacy class $C$ such that for any $x,y\in T \setminus \{1\}$, there exists $z\in C$ with $\la x,z\ra = \la y,z\ra = T$. Let $x_1,y_1\in T$ and $x_2,y_2 \in C$ be such that: $|x_1|,|y_1|$ and $|x_2|$ 
 are distinct; $x_2$ and $y_2$ are distinct; and 
 $\la x_1,x_2\ra = \la y_1,y_2\ra = T$. To do so, we can first choose an appropriate $x_1$ and $y_1$, then find $x_2 \in C$ such that $\langle x_1,x_2 \rangle = \langle y_1,x_2 \rangle = T$, and finally replace $y_1$ with $y_1^t$ for some $t \in T$ such that $y_2:=x_2^t \ne x_2$. Note that the setwise stabilisers of $\{x_1,x_2\}$ and $\{y_1,y_2\}$ in $\Aut(T)$ are trivial, and that $\{x_1,x_2\}^\a \ne \{y_1,y_2\}$ for any $\a\in\Aut(T)$.
		
 Assume now that $k = |T|^\ell-m$, with $m \in \{0,1,2\}$. Define $\mathbf{x}\in\Inn(T)^k$ by setting $|\mathcal{P}_1^{\mathbf{x}}| = |T|^{\ell-1}+1$, $|\mathcal{P}_{x_1}^{\mathbf{x}}| = |T|^{\ell-1}-1$, $|\mathcal{P}_{x_2}^{\mathbf{x}}| = |T|^{\ell-1}-m$, and $|\mathcal{P}_t^{\mathbf{x}}| = |T|^{\ell-1}$ for $t\in T\setminus\{1,x_1,x_2\}$. We similarly define $\mathbf{y}\in\Inn(T)^k$, with $y_1$ and $y_2$ in place of $x_1$ and $x_2$, respectively. Then as noted in the proof of \cite[Lemma 5.10]{H_diag}, $D$ is adjacent in $\Sigma(G)$ to both $D\mathbf{x}$ and $D\mathbf{y}$.
 
 Suppose that $D\mathbf{x}$ and $D\mathbf{y}$ lie in a common $G$-orbit. Since
	\begin{equation*}
	\{t:|\mathcal{P}_t^{\mathbf{x}}| = |T|^{\ell-1}+1\} = \{1\} = \{t:|\mathcal{P}_t^{\mathbf{y}}| = |T|^{\ell-1}+1\},
	\end{equation*}
	the element $g\in T$ described in Lemma \ref{l:diag_subaction} is the identity. Thus if $m$ = 1, then that lemma yields the existence of some $\a \in \Aut(T)$ such that
	\begin{equation*}
	\{x_1,x_2\}^\a = \{t:|\mathcal{P}_t^{\mathbf{x}}| = |T|^{\ell-1}-1\}^\a = \{t:|\mathcal{P}_t^{\mathbf{y}}| = |T|^{\ell-1}-1\} = \{y_1,y_2\},
	\end{equation*}
	which is incompatible with our assumption. If instead $m \in \{0,2\}$, then we similarly observe that $x_1^\a = y_1$ for some $\a \in \Aut(T)$, another contradiction. Therefore, $\Sigma(G)$ is not $G$-arc-transitive.
\end{proof}

\begin{prop}
	\label{p:diag_arc-trans_|T|ell-2-1-0}
	Suppose that $P\in\{A_k,S_k\}$, and that $k\in\{|T|^\ell-2,|T|^\ell-1,|T|^\ell\}$ for some $\ell\geqs 1$. Then $\Sigma(G)$ is not $G$-arc-transitive, and hence $\reg(G)\geqs 2$.
\end{prop}

\begin{proof}
    We first note that if $b(G) = \ell+2$, then Proposition~\ref{p:diag_l+2_complete} shows that $\Sigma(G)$ is complete, and therefore not $G$-arc-transitive (since $G$ is not $2$-transitive). In particular, by Theorem~\ref{t:H_diag}, this holds if $k = |T|$, or if $S_k\leqs G$ and $k \ne |T|^\ell-2$. If instead $k \in \{|T|-2,|T|-1\}$ and $S_k \not \leqs G$, then Theorem~\ref{t:H_diag} yields $b(G) = 2$, and so the desired result follows from Theorem~\ref{t:H_diag_r=1}. Additionally, one can simply apply Lemma~\ref{l:diag_arc-trans_notSk} to the groups with $\ell\geqs 2$ and $S_k\not\leqs G$.

    To complete the proof, we shall assume that $k = |T|^\ell-2$, $b(G) \ne \ell+2$ and $S_k\leqs G$, so that $G = T^k.(O\times S_k)$ for some $O\leqs \Out(T)$. Suppose first that $\ell = 2$, so that $b(G) \ne 4$. Then it follows from Theorem~\ref{t:H_diag} that $b(G) = 3$, and that $O\ne\Out(T)$ if $T\in\{A_5,A_6\}$. Let $(x_1,x_2)$ be a pair of elements of $T$ as described in Lemma \ref{l:diag_arc-trans_star}, and define $\mathbf{x}_1,\mathbf{x}_2\in\Inn(T)^k$ by setting $|\mathcal{P}_1^{\mathbf{x}_i}| = |\mathcal{P}_{x_i}^{\mathbf{x}_i}| = |T|-1$ and $|\mathcal{P}_t^{\mathbf{x}_i}| = |T|$ for $t\in T\setminus\{1,x_i\}$. Then as noted in the proof of \cite[Proposition 5.13]{H_diag}, $D$ and $D\mathbf{x}_i$ are adjacent in $\Sigma(G)$. It follows readily from Lemma~\ref{l:diag_subaction} that $(D,D\mathbf{x}_1)$ and $(D,D\mathbf{x}_2)$ lie in distinct $G$-orbits.

	Finally, suppose that $\ell\geqs 3$. Let $x\in T$, and (similarly to above) let $\mathbf{x}\in\Inn(T)^k$ be such that $|\mathcal{P}_1^{\mathbf{x}}| = |\mathcal{P}_x^{\mathbf{x}}| = |T|^{\ell-1}-1$ and $|\mathcal{P}_t^{\mathbf{x}}| = |T|^{\ell-1}$ for $t\in T\setminus\{1,x\}$. Then as observed in the proof of \cite[Proposition 5.14]{H_diag}, $D$ and $D\mathbf{x}$ are adjacent in $\Sigma(G)$. Now apply Lemma \ref{l:diag_subaction}, noting that $x$ was chosen arbitrarily.
\end{proof}

We are now in a position to prove Theorem \ref{thm:diag_arc-trans}.

\begin{proof}[Proof of Theorem \ref{thm:diag_arc-trans}.]
	First note by Theorem \ref{t:H_diag_r=1} that if $b(G) = 2$, then $\reg(G) = 1$ if and only if $G$ is one of the two special cases arising in the statement of Theorem \ref{thm:diag_arc-trans}. Thus, in view of Theorem \ref{t:H_diag}, we may assume that $P\in\{A_k,S_k\}$, and either $k = 2$ or $k \ge |T|-2$. The result follows from Proposition~\ref{p:diag_arc-trans_k=2} in the first case, and from Propositions \ref{p:diag_arc-trans_l+1} and \ref{p:diag_arc-trans_|T|ell-2-1-0} in the second case.
 \end{proof}

\section{Almost simple groups}
\label{s:as}

In this section, we investigate Conjecture~\ref{conj:BG} in the context of almost simple primitive groups, as well as the question of which of these groups are semi-Frobenius or have $G$-arc-transitive generalised Saxl graphs. Our focus is on groups with sporadic socle, groups with socle $\LL_2(q)$, and groups with soluble point stabilisers. In particular, we prove Theorems~\ref{thm:asconj}, \ref{thm:PSL2_semi} and \ref{introthm:PSL2_arc}. Along the way, we complete the calculation of base sizes for all primitive groups with socle $\LL_2(q)$.

\subsection{Almost simple sporadic groups}

\begin{prop}
\label{prop:sporsemi}
	Let $G$ be a primitive almost simple group with sporadic socle $G_0$ and point stabiliser $H$, such that $b(G) \ge 3$.\begin{itemize}\addtolength{\itemsep}{0.2\baselineskip}
		\item[{\rm (i)}] If $b(G) \ge 5$, then $G$ is semi-Frobenius.
  \item[{\rm (ii)}] If $G_0 \notin \{\mathrm{Co}_1,\mathrm{J}_4,\mathrm{Fi}_{24}',\mathbb{B}, \mathbb{M}\}$, then $G$ is semi-Frobenius if and only if $(G,H,b(G))$ does not appear in Table~\ref{tab:spor_not_complete}.
 \end{itemize}
 \end{prop}
 
{\small
	\begin{table}[h!]
       
       \caption{Primitive almost simple sporadic groups that are not semi-Frobenius. When $G = \mathrm{M}_{12}.2$, exactly one of the two groups $H \cong \LL_2(11).2$, up to conjugacy, satisfies $H \cap G_0 \underset{\max}{<} G_0$.}
		\begin{tabular}{@{}lll@{}}
			\toprule
			$G$ & $H$ & $b(G)$ \\ \midrule
			$\mathrm{M}_{12}$ & $\LL_2(11)$ & 3 \\
			&  &  \\
			$\mathrm{M}_{12}.2$ & $\LL_2(11).2$ & 3 \\
            & $\LL_2(11).2$ & 3 \\
			&  &  \\
			$\mathrm{M}_{22}$ & $2^4 \sd S_5$ & 3 \\
			& $2^3 \sd  \LL_3(2)$ & 3 \\
			&  &  \\
			$\mathrm{M}_{22}.2$ & $2^5 \sd S_5$ & 3 \\
			& $2^3 \sd  \LL_3(2)\times 2$ & 3 \\
			&  &  \\
			$\mathrm{M}_{23}$ & $A_8$ & 3 \\
			& $2^4 \sd (3\times A_5) \sd 2$ & 3 \\
			&  &  \\
			$\mathrm{M}_{24}$ & $\mathrm{M}_{22}.2$ & 4 \\
			& $\mathrm{M}_{12}.2$ & 3 \\
			& $2^6 \sd 3.S_6$ & 3 \\
			& $\LL_3(4) \sd S_3$ & 3 \\
			&  &  \\
			$\mathrm{J}_2$ & $3.A_6.2$ & 3 \\
			& $2^{1+4}.A_5$ & 3 \\
			&  &  \\
			$\mathrm{J}_2.2$ & $3.A_6.2.2$ & 3 \\
			& $2^{1+4}.A_5.2$ & 3 \\
			& $2^{2+4} \sd (3\times S_3).2$ & 3 \\
            & & \\
			\bottomrule
		\end{tabular}
		\begin{tabular}{@{}lll@{}}
			\toprule
			$G$ & $H$ & $b(G)$ \\ \midrule
            $\mathrm{McL}$ & $3^4 \sd \mathrm{M}_{10}$ & 3 \\
            & & \\
            $\mathrm{McL}.2$ & $3^4 \sd (\mathrm{M}_{10}\times 2)$ & 3 \\
            &  &  \\
            $\mathrm{HS}$ & $\LL_3(4) \sd 2$ & 3 \\
			& $S_8$ & 3 \\
            &  &  \\
			$\mathrm{HS}.2$ & $2^5.S_6$ & 3 \\
			& $4^3.(2\times \LL_3(2))$ & 3 \\
			&  &  \\
			$\mathrm{Co}_3$ & $\UU_4(3).2.2$ & 3 \\ & $\mathrm{M}_{23}$ & 3 \\
			&  &  \\
			$\mathrm{Co}_2$ & $\mathrm{HS} \sd 2$ & 3 \\
			& $(2^4\times 2^{1+6}).A_8$ & 3 \\
			& $\UU_4(3) \sd D_8$ & 3 \\
			&  &  \\
			$\mathrm{Fi}_{22}$ & $2^{10} \sd \mathrm{M}_{22}$ & 3 \\
			& $2^6 \sd \Sp_6(2)$ & 3 \\
			& $(2\times 2^{1+8}) \sd (\UU_4(2) \sd 2)$ & 3 \\
			&  &  \\
			$\mathrm{Fi}_{22}.2$ & $2^7 \sd \Sp_6(2)$ & 3 \\
			& $(2\times 2^{1+8} \sd \UU_4(2) \sd 2) \sd 2$ & 3 \\
			& $\UU_4(3).2.2\times S_3$ & 3 \\
			& $2^{5+8} \sd (S_3\times S_6)$ & 3 \\
            & & \\
            \bottomrule
		\end{tabular}
		\begin{tabular}{@{}lll@{}}
			\toprule
			$G$ & $H$ & $b(G)$ \\ \midrule
            $\mathrm{HN}$ & $A_{12}$ & 3 \\
			& $2.\mathrm{HS}.2$ & 3 \\
			&  &  \\
            $\mathrm{HN}.2$ & $S_{12}$ & 3 \\
			& $4.\mathrm{HS}.2$ & 3 \\
            &  &  \\
            $\mathrm{He}$ & $2^2.\LL_3(4).S_3$ & 3 \\
            &  &  \\
            $\mathrm{He}.2$ & $2^2.\LL_3(4).D_{12}$ & 3 \\
            &  &  \\
			$\mathrm{Suz}$ & $G_2(4)$ & 4 \\
			& $3.\UU_4(3) \sd 2$ & 3 \\
			& $\UU_5(2)$ & 3 \\
			& $2^{1+6}.\UU_4(2)$ & 3 \\
			& $3^5.\mathrm{M}_{11}$ & 3 \\
			&  &  \\
			$\mathrm{Suz}.2$ & ${G}_2(4).2$ & 4 \\
			& $3.\UU_4(3).2.2$ & 3 \\
			& $\UU_5(2) \sd 2$ & 3 \\
			& $2^{1+6}.\UU_4(2).2$ & 3 \\
			& $3^5.(\mathrm{M}_{11}\times 2)$ & 3 \\
			&  &  \\
			$\mathrm{Fi}_{23}$ & $\mathrm{P}\Omega_8^+(3).S_3$ & 4 \\
			& $2^2.\UU_6(2).2$ & 3 \\
			& $\Sp_8(2)$ & 3 \\
			& $2^{11}.\mathrm{M}_{23}$ & 3 \\
            \bottomrule
		\end{tabular}
		 \label{tab:spor_not_complete}
	\end{table}
}
\begin{proof}
    The almost simple primitive groups with sporadic socle of each base size at least $3$ are classified in \cite{BOW_spor,NNOW_spor}. In particular, if $b(G) \ge 5$ and $G_0 \in \mathcal{A}:=\{\mathrm{Co}_1,\mathrm{J}_4,\mathrm{Fi}_{24}',\mathbb{B}, \mathbb{M}\}$, then $G_0 \in \{\mathrm{Co}_1,\mathrm{Fi}_{24}'\}$, and $H$ is a maximal subgroup of $G$ of least index. Since $G$ is transitive, it is semi-Frobenius if and only if, for a fixed $\omega \in \Omega$ and each point $\alpha \ne \omega$ in a set of orbit representatives of $G_\omega$, the set $\{\omega,\alpha\}$ extends to a base for $G$ of size $b(G)$. For all groups $G$ satisfying $G_0 \notin \mathcal{A}$ or $b(G) \ge 5$, except for one special case mentioned at the end of the proof, we use {\sc Magma} to directly check this condition on orbit representatives.
    
    We now detail how we construct the group $G$ in {\sc Magma}. If $G_0 \in \{\mathrm{Ly},\mathrm{Th}\}$, then we construct a (relatively) low-dimensional matrix group $\hat{G}$ isomorphic to $G$ using {\sc Magma}'s database of {\sc Atlas} groups, and a subgroup $\hat{H}$ isomorphic to $H$ using generators from \cite{onlineatlas}. Using the method described in \cite{magmaorbit}, we then obtain $G$ as the permutation group induced by the action of $\hat{G}$ on the orbit $U^{\hat{G}}$, where $U$ is a low-dimensional $\hat{H}$-submodule of the module corresponding to $\hat{G}$ (this construction requires a significant amount of RAM, e.g.~$171$ GB when $G = \mathrm{Th}$ and $H \cong 2^5.\LL_5(2)$). If instead $G_0 \notin \{\mathrm{Ly},\mathrm{Th}\}$, then we construct a permutation group $\hat{G}$ isomorphic to $G$ using the database of {\sc Atlas} groups (or if $G = \mathrm{HN}.2$, using the \texttt{AutomorphismGroupSimpleGroup} function), and then construct $G$ itself via the coset action of $\hat{G}$ on a subgroup $\hat{H}$ isomorphic to $H$. 

    The one special case mentioned above is where $G = \mathrm{Fi}_{23}$ and $H \cong 3^{1+8}.2^{1+6}.3^{1+2}.2S_4$, with $b(G) = 3$. Here, the index of $H$ in $G$ is too large for the permutation group $G$ to be constructed directly. Therefore, with $\hat{G}$ and $\hat{H}$ permutation groups as above, we search through the elements of $\hat{G}$ for a full set of representatives of the double cosets in the set $\hat{H}\backslash \hat{G}/\hat{H}$, with the aid of the \texttt{DoubleCosetCanonical} function. For each representative $x \notin \hat{H}$, we verify that there exists an element $y \in \hat{G}$ such that $\hat{H} \cap \hat{H}^x \cap \hat{H}^y = 1$. This implies the above condition on orbit representatives of $G_\omega$, i.e.~$G$ is semi-Frobenius.
\end{proof}

\begin{rem}
\label{rem:otherspor}
Using computational methods from the proof of Proposition~\ref{prop:sporsemi}, we can show that $G$ is semi-Frobenius if $G_0 = \mathrm{Co}_1$ and $b(G) = 4$ (so that $H \cong 3.\mathrm{Suz} \sd 2$), and that $G$ is not semi-Frobenius if $G_0 \in \{\mathrm{Co}_1,\mathrm{J}_4\}$, $b(G) = 3$, and $|\Omega| \le 2\times10^8$. Determining which of the remaining primitive groups with $G_0 \in \{\mathrm{Co}_1,\mathrm{J}_4,\mathrm{Fi}_{24}',\mathbb{B}, \mathbb{M}\}$ are semi-Frobenius will likely require alternative approaches, due to the extremely large degrees of these groups.
\end{rem}
In \cite[Section 6]{BG_Saxl}, Burness and Giudici show that many almost simple sporadic groups $G$ with $b(G) = 2$ satisfy Conjecture~\ref{conj:BG}. We now generalise their results (and Proposition~\ref{prop:sporsemi}) in the case $b(G) \ge 3$, as follows.

\begin{thm}
\label{thm:sporcnc}
Let $G$ be a primitive almost simple group with sporadic socle, such that $b(G) \ge 3$. Then $G$ satisfies Conjecture~\ref{conj:BG}.
\end{thm}

\begin{proof}
We first observe that any two vertices in $\Sigma(G)$ have a common neighbour if and only if, for a fixed $\omega \in \Omega$, each point $\alpha \ne \omega$ in a set of orbit representatives of $G_\omega$ is adjacent to a neighbour of $\omega$. In particular, this is the case if the valency $\mathrm{val}(G)$ of $\Sigma(G)$ satisfies $\mathrm{val}(G)/|\Omega| > 1/2$.

Now, by Proposition~\ref{prop:sporsemi}, it suffices to consider the case where either $G$ and $H$ appear in Table~\ref{tab:spor_not_complete}, or $\soc(G) \in \{\mathrm{Co}_1,\mathrm{J}_4,\mathrm{Fi}_{24}',\mathbb{B},\mathbb{M}\}$ and $b(G) \in \{3,4\}$. The relevant groups with these five socles, and their base sizes, are given in \cite{BOW_spor,NNOW_spor}.

If $G$ and $H$ appear in Table~\ref{tab:spor_not_complete}, or if $G = \mathrm{Co}_1$ and $H$ is isomorphic to $\mathrm{2^{11} \sd \mathrm{M}_{24}}$ or $\mathrm{Co}_3$, then we construct the permutation group $G$ in {\sc Magma} via the method described in Proposition~\ref{prop:sporsemi}, and verify the above condition on orbit representatives. In particular, we observe that among these groups, $\mathrm{val}(G)/|\Omega| \le 1/2$ if and only if $G = \mathrm{M}_{12}.2$, $H \cong \LL_2(11).2$, and $H \cap G_0$ is maximal in $G_0$. Even in this case, computations show that any two vertices in $\Sigma(G)$ have a common neighbour.

Next, suppose that $G = \mathrm{Fi}_{24}$ and $H \cong (2 \times 2.\mathrm{Fi}_{22}) \sd 2$, so that $b(G) = 3$. We construct a permutation group $\hat G$ isomorphic to $G$ in {\sc Magma} using the database of {\sc Atlas} groups, and a subgroup $\hat H$ isomorphic to $H$ using generators from \cite{onlineatlasv2}. Computations show that there exist elements $r,s \in \hat G \setminus \hat H$ such that $\hat H \cap \hat H^r \cap \hat H^s = 1$ and $|\hat H r \hat H|/|\hat G| > 1/2$. Thus $\omega \in \Omega$ is adjacent in $\Sigma(G)$ to each point in a $G_\omega$-orbit of size $a:=|\hat H r \hat H|/|\hat H|$, and so $\mathrm{val}(G)/|\Omega| \ge a/|\Omega| = |\hat H r \hat H|/|\hat G| > 1/2$.

In all remaining cases, let $R(G)$ be a set of representatives for the $G$-conjugacy classes of elements of $H$ of prime order. Recall from Section~\ref{ss:pre_prob} that if \[\hat{Q}(G,b(G)) = \sum_{x\in R(G)}\frac{|x^G\cap H|^{b(G)}}{|x^G|^{{b(G)}-1}} < 1/2,\] then any two vertices in $\Sigma(G)$ have a common neighbour. In each case, we verify that $\hat{Q}(G,b(G)) < 1/2$ using the {\sf GAP} Character Table Library \cite{GAPchar}. Note that if $G = \mathbb{B}$ and $H \cong (2^2 \times F_4(2)) \sd 2$, then the fusion of $H$-conjugacy classes in $G$ is not stored in {\sf GAP}. However, all possibilities for this fusion, as returned via the \texttt{PossibleClassFusions} function, yield the same value for $\hat{Q}(G,b(G))$.
\end{proof}

\subsection{Groups with socle $\LL_2(q)$}
\label{subsec:l2q}

In this subsection, we provide a complete classification of groups $G$ with $\soc(G) = \LL_2(q)$ that are semi-Frobenius, and a near-complete classification of such groups for which $\Sigma(G)$ is $G$-arc-transitive.

We shall begin by determining the base size of a certain primitive action of $G$. To do so, we require the following technical lemma. In what follows, for an element $t \in \mathbb{F}_q$, we write $\langle t \rangle$ to denote the smallest subfield of $\mathbb{F}_q$ containing $t$.

\begin{lem}
\label{lem:fieldgen}
Let $q$ be a prime power not equal to $3$, and let $s$ be a primitive element of $ \mathbb{F}_q$. %Then there exists an element $t \in \mathbb{F}_q^\times$ such that:
 \begin{itemize}\addtolength{\itemsep}{0.2\baselineskip}
		\item[{\rm (i)}] If $q$ is even, then there exists $t \in \mathbb{F}_q^\times$ such that $\langle t \rangle = \mathbb{F}_q$, and the polynomial $x^2+x+t$ is irreducible over $\mathbb{F}_q$.
        \item[{\rm (ii)}] If $q$ is odd, then there exists $t \in \mathbb{F}_q^\times$ such that $\langle t^2/s \rangle = \mathbb{F}_q$, and $t^2+s$ is a (non-zero) square in $\mathbb{F}_q$; and
        \item[{\rm (iii)}] If $q \equiv 3 \pmod 4$, then there exists $t \in \mathbb{F}_q^\times$ such that $\langle t^2 \rangle = \mathbb{F}_q$, and $t^2+1$ is a (non-zero) square in $\mathbb{F}_q$.
	\end{itemize}
\end{lem}

\begin{proof}
We shall write $q = p^f$, with $p$ prime. Suppose first that $q$ is even, and let $r$ be a divisor of $f$. For $a \in \mathbb{F}_{2^r}$, the polynomial $x^2+x+a$ is reducible over $\mathbb{F}_q$ if and only if there exists $b \in \mathbb{F}_q$ such that $b^2+b = a$, so that the polynomial has distinct roots $b$ and $b+1$. Hence there are precisely $2^{r-1}$ elements $a \in \mathbb{F}_{2^r}$ such that $x^2+x+a$ is irreducible over $\mathbb{F}_{2^r}$. Additionally, $x^2+x+0$ is reducible over $\mathbb{F}_q$, as is each polynomial reducible over $\mathbb{F}_{2^r}$. Therefore, the number of elements $a \in \mathbb{F}_q^\times$ that lie in a proper subfield of $\mathbb{F}_q$, with $x^2+x+a$ irreducible over $\mathbb{F}_q$, is at most
\[\sum_{\substack{1 \le r < f\\r \mid f}} 2^{r-1} < 2^{f-1}.\]
Therefore, among the $2^{f-1}$ elements $t \in \mathbb{F}_q^\times$ such that $x^2+x+t$ is irreducible, there exists at least one such that $\langle t \rangle = \mathbb{F}_q$, and hence $t$ satisfies (i).

Now, if $q$ is odd, then since $s$ is not a square in $\mathbb{F}_q$, exactly one of $1+s$ and $s(1+s) = s^2+s$ is a square (these are both nonzero as $q > 3$). Thus some $t \in \{1,s\}$ satisfies (ii). If, in addition, $q \equiv 3 \pmod 4$, then it follows from \cite[Theorem 3.1]{circle} and a simple counting argument that the number of squares in the set $\{c^2+1 \mid c \in \mathbb{F}_{q}^\times\}$ is $(q-3)/4$, which is nonzero since $q > 3$. Notice that $f$ is odd, and so $\langle m^2 \rangle = \langle m \rangle$ for all $m \in \mathbb{F}_{q}$. Therefore, the number of squares in $\mathbb{F}_{q}^\times$ that lie in a proper subfield of $\mathbb{F}_{q} = \mathbb{F}_{p^{f}}$ is at most
\[\sum_{\substack{1 \le r < f\\r \mid f}} (p^r-1)/2 < (p^{f-1}-1)/2 < (p^{f}-3)/4 = (q-3)/4.\]
Thus there exists $t \in \mathbb{F}_{q}^\times$ such that $t^2+1$ is a square, and such that $\langle t^2 \rangle = \langle t \rangle = \mathbb{F}_q$. Note finally that $-1$ has no square root in $\mathbb{F}_q$ since $q \equiv 3 \pmod 4$, and so $t$ satisfies (iii).
\end{proof}

\begin{prop}
\label{prop:basesizeq0}
Let $G$ be an almost simple primitive group with socle $G_0 = \LL_2(q)$ and point stabiliser $H$ of type $\GL_2(q_0)$, with $q = q_0^2$ for some $q_0 \ge 3$. If $G = \PSigmaL_2(9)$, then $b(G) = 4$, and otherwise $b(G) = 3$.
\end{prop}

\begin{proof}
It is straightforward to use {\sc Magma} to calculate $b(G)$ in the case $q_0 = 3$. Assume therefore that $q_0 \ge 4$. By \cite[p.~354]{FI_subdeg}, $G_0$ has no regular suborbits in its action on the right cosets of $H\cap G_0$. Therefore, $b(G) \geqs b(G_0)\geqs 3$. Additionally, the maximality of $H$ implies that $G\leqs \PSigmaL_2(q)$. To complete the proof, we will assume that $G = \PSigmaL_2(q)$ and show that $b(G) = 3$. We shall write $q = p^f$, with $p$ prime.

It will be convenient to identify the action of $G$ on the right cosets of $H$ with an equivalent action on certain $1$-dimensional subspaces of $\mathbb{F}_{q_0}^4$, corresponding to the isomorphism $G_0 \cong \Omega^-_4(q_0) \cong \mathrm{P}\Omega^-_4(q_0)$ (see \cite[Proposition 2.9.1]{kleidmanliebeck}). To define this action, let $s$ be a primitive element of $\mathbb{F}_{q_0}$, and let $Q$ be a non-degenerate quadratic form of minus type on $V:=\mathbb{F}_{q_0}^4$, with polar form $\beta$. By \cite[Propositions 1.5.39 \& 1.5.42]{BHRD} and \cite[Proposition 2.5.12]{kleidmanliebeck}, there exists a basis $\{e_1,e_2,e_3,e_4\}$ for $V$ such that the matrix of $Q$ (defined so that the $(i,j)$ entry is equal to $\beta(e_i,e_j)$ if $i < j$, to $Q(e_i)$ if $i = j$, and to $0$ otherwise) is
\[
M_Q := \begin{pmatrix}
    0&0&0&1\\
    0 & 1 & 1& 0\\
    0&0& \zeta &0\\
    0&0&0&0\\
\end{pmatrix},
\]
if $q$ is even and 
\[
M_Q := \begin{pmatrix}
    s/2&0&0&0\\
    0 & 1/2 & 0& 0\\
    0&0& 1/2 &0\\
    0&0&0&1/2\\
\end{pmatrix},
\]
if $q$ is odd, where $\zeta$ is an arbitrary element of $\mathbb{F}_{q_0}^\times$ such that the polynomial $x^2+x+\zeta$ is irreducible over $\mathbb{F}_{q_0}$. By Lemma~\ref{lem:fieldgen}, we may assume that $\langle \zeta \rangle = \mathbb{F}_{q_0}$. The Gram matrix $M_\beta$ of $\beta$ is equal to $M_Q + M_Q^T$.

Next, define $\sigma : V \to V$ by 
\[
\sigma: \sum_{i=1}^4 a_i e_i \mapsto a_3^2\zeta e_2  + \sum_{i=1}^4 a_i^2e_i,
\]
if $q$ is even, and
\[
\sigma: \sum_{i=1}^4 a_i e_i \mapsto a_1^p s^{(p-1)/2} e_1 + \sum_{i=2}^4 a_i^pe_i,
\]
if $q$ is odd. Hence for each positive integer $k$, the map $\sigma^k$ is defined by
\[
\sigma^k: \sum_{i=1}^4 a_i e_i \mapsto a_3^{2^k}\sum_{j=0}^{k-1} \zeta^{2^j} e_2  + \sum_{i=1}^4 a_i^{2^k}e_i,
\]
if $q$ is even, and
\[
\sigma^k: \sum_{i=1}^4 a_i e_i \mapsto a_1^{p^k} s^{(p^k-1)/2} e_1 + \sum_{i=2}^4 a_i^{p^k}e_i,
\] 
if $q$ is odd.
Additionally, $|\sigma| = f$, and $\sigma$ is induced by an automorphism of the simple derived subgroup $\Omega^-_4(q_0)$ of the isometry group $\mathrm{GO}^-_4(q_0)$ of $Q$ (see \cite[pp.~58--59]{BG_classical}). Slightly abusing notation and denoting this automorphism by $\sigma$, we may identify $G$ with $\langle \Omega^-_4(q_0), \sigma \rangle/\langle -I \rangle$, where $I$ is the $4 \times 4$ identity matrix (cf.~\cite[Tables 8.1 \& 8.17]{BHRD}). Note that $L:=\langle \Omega^-_4(q_0), \sigma \rangle = \langle \sigma \rangle \mathrm{GO}^-_4(q_0)$, and that $\langle \sigma \rangle \cap \mathrm{GO}^-_4(q_0)$ is generated by $\sigma^{f/2}$. Let $Y$ be a one-dimensional subspace of $V$ and $y \in Y \setminus \{0\}$. We write $Y \in \Delta$ if $Q(y)$ is a non-zero square in $\mathbb{F}_{q_0}$, and $Y \in \overline{\Delta}$ if $Q(y)$ is a non-square. The action of $G$ on the right cosets of $H$ is equivalent to its action on $\Delta$, and also to its action on $\overline{\Delta}$ if $q$ is odd.

Suppose now that $q$ is even, and let $u:=e_2$, $v:=e_1 + e_2$ and $w:= e_3$. Then $Q(u) = Q(v) = 1$ and $Q(w) = \zeta$, and so $\langle u \rangle, \langle v \rangle, \langle w \rangle \in \Delta$. By \cite[Lemma 1.5.21]{BHRD}, a matrix $A \in \GL_4(q_0)$ lies in $\mathrm{GO}^-_4(q_0)$ if and only the diagonal entries of $A M_Q A^T$ are equal to the corresponding diagonal entries of $M_Q$, and $A M_\beta A^T =M_\beta$. Therefore, the pointwise stabiliser $X$ of $\{\langle u \rangle, \langle v \rangle\}$ in $L$ consists of all elements of the form
\[
\sigma^k  \begin{pmatrix}
     1&0&0&0\\
     0&1 &  0& 0\\
    a&b& 1 &0\\
    a^2&a&0&1\\
\end{pmatrix},
\]
where $a\in \mathbb{F}_{q_0}^{\times}$, $b \in \{0,1\}$ and $k \in \{0,\ldots,f/2-1\}$. We also observe that an element of $X$ fixes $\langle w \rangle$ if and only if $a = 0$ and $\sum_{j=0}^{k-1} \zeta^{2^j} = b$. Since $b = b^2$, this implies that
\[\sum_{j=0}^{k-1} \zeta^{2^j} = {\left(\sum_{j=0}^{k-1} \zeta^{2^j}\right)\!\!}^2 = \sum_{j=0}^{k-1} (\zeta^{2^j})^2 = \sum_{j=1}^{k} \zeta^{2^j},\]
and hence $\zeta = \zeta^{2^k}$. Since $\langle \zeta \rangle = \mathbb{F}_{q_0}$ by our assumption above, it follows that $k = 0$, and so $b = 0$. Thus $\{\langle u \rangle, \langle v \rangle, \langle w \rangle\}$ is a base for $G$, and hence $b(G) = 3$.

Next, suppose that $q$ is odd, let $\delta \in \{1,3\}$ such that $q_0 \equiv \delta \pmod 4$, and let $\gamma_1:=s$ and $\gamma_3 := 1$. By Lemma~\ref{lem:fieldgen}, there exists $t \in \mathbb{F}_{q_0}^\times$ such that $t^2+\gamma_\delta$ is a (non-zero) square in $\mathbb{F}_{q_0}$ and $\langle t^2/\gamma_\delta \rangle = \mathbb{F}_{q_0}$. Define $u := e_2$, $v := e_\delta+t e_2$ and $w:=e_2+ze_4+e_{4-\delta}$, where  $z:=s^{(q_0-2+\delta)/4}$. Then $Q(v) = (t^2+\gamma_\delta)/2$, and $Q(w) = (1+z^2+\gamma_{4-\delta})/2 = 1/2 = Q(u)$. Thus $\langle u \rangle, \langle v \rangle, \langle w \rangle \in \Gamma$ for some $\Gamma \in \{\Delta,\overline{\Delta}\}$. In this case, a matrix $A \in \GL_4(q_0)$ lies in $\mathrm{GO}^-_4(q_0)$ if and only if $A M_\beta A^T =M_\beta$. Using the fact that $\langle t^2/\gamma_\delta \rangle = \mathbb{F}_{q_0}$, we deduce that the pointwise stabiliser of $\{\la u \ra, \la v \ra\}$ in $L$ lies in $\mathrm{GO}^-_4(q_0)$, and consists of all matrices of the form
\[
\begin{pmatrix}
     a&0&0&0\\
    0 &a &  0& 0\\
    0&0& b_1 &c_1\\
    0&0&\varepsilon c_1&-\varepsilon b_1\\
\end{pmatrix},
\] 
if $\delta = 1$,
or
\[
\begin{pmatrix}
     b_2&0&0&c_2\\
    0 &a &  0& 0\\
    0&0& a &0\\
    -\varepsilon c_2 s^{-1}&0&0&\varepsilon b_2\\
\end{pmatrix},
\]
if $\delta = 3$, where
$a,\varepsilon \in \{\pm 1\}$ and $b_1^2+c_1^2 = b_2^2 + c_2^2 s^{-1} = 1$. It is now straightforward to show that the pointwise stabiliser of $\{\la u \ra, \la v \ra, \la w \ra\}$ in $L$ consists of scalar matrices. Therefore, $b(G) = 3$.
\end{proof}

By combining Proposition~\ref{prop:basesizeq0} with results from the literature, we are able to determine the exact base size of every primitive group with socle $\LL_2(q)$. Here we exclude the groups with socle $\LL_2(4)$, as $\LL_2(4)\cong \LL_2(5)$.

\begin{thm}
    \label{t:PSL2_b(G)}
    Let $G$ be a primitive group with socle $G_0=\LL_2(q)$, $q\geqs 5$ and point stabiliser $H$. Then $b(G) > 2$ if and only if one of the following holds.
    \begin{itemize}\addtolength{\itemsep}{0.2\baselineskip}
		\item[{\rm (i)}] $H$ is of type $P_1$, in which case $b(G)\in\{3,4\}$. Furthermore, $b(G) = 3$ if and only if either $G\leqs\PGL_2(q)$, or $|G:G_0| = 2$ and $G \not\le \PSigmaL_2(q)$.
        \item[{\rm (ii)}] $H$ is of type $\GL_1(q)\wr S_2$ and $\PGL_2(q) < G$, in which case $b(G) = 3$.
        \item[{\rm (iii)}] $H$ is of type $\GL_1(q^2)$ and $\PGL_2(q)\leqs G$, in which case $b(G) = 3$.
        \item[{\rm (iv)}] $H$ is of type $\GL_2(q_0)$ with $q = q_0^2 \ge 9$, in which case $b(G) \in \{3,4\}$. Furthermore, $b(G) = 4$ if and only if $G = \PSigmaL_2(9)$.
        \item[{\rm (v)}] $(q,H) \in \{(5,A_4),(7,S_4), (11,A_5), (19,A_5)\}$ and $b(G) = 3$, $(q,H)\in\{(5,S_4),(9,A_5)\}$ and $b(G) = 4$, or $(q,H) = (9,S_5)$ and $b(G) = 5$.
        
	\end{itemize}
\end{thm}

\begin{proof}
    If $H$ is soluble, then \cite[Theorem 2]{B_sol} gives the precise base size of $G$. This includes the cases where $H$ is of type $P_1$, $\GL_1(q)\wr S_2$, $\GL_1(q^2)$ or $2_-^{1+2}.\Omega_2^-(2)$. For the groups with $H$ of type $A_5$, or of type $\GL_2(q_0)$ with $q_0^k = q$ for some odd prime $k$, we deduce the base size from \cite[Tables 1 \& 3]{B_class}. Suppose finally that $H$ is of type $\GL_2(q_0)$ with $q_0^2 = q$, noting that if $q_0 = 2$, then $H$ is considered as a group of type $\GL_1(q)\wr S_2$. Here, we obtain $b(G)$ from Proposition~\ref{prop:basesizeq0}.
\end{proof}
We are now able to classify the almost simple primitive groups with socle $\LL_2(q)$ that are semi-Frobenius. This establishes Theorem \ref{thm:PSL2_semi} stated in Section \ref{s:intro}.

\begin{thm}
    \label{t:PSL2_semi}
    Let $G \leq \mathrm{Sym}(\Omega)$ be an almost simple primitive group with socle $G_0 = \LL_2(q)$ and point stabiliser $H$. Then $G$ is semi-Frobenius if and only if one of the following holds:
    \begin{itemize}\addtolength{\itemsep}{0.2\baselineskip}
		\item[{\rm (i)}] $H$ is of type $P_1$;
        \item[{\rm (ii)}] $H$ is of type $\GL_1(q)\wr S_2$ and $\PGL_2(q) < G$;
        \item[{\rm (iii)}] $H$ is of type $\GL_1(q^2)$ and $\PGL_2(q)\leqs G$;
		\item[{\rm (iv)}] $H$ is of type $\GL_2(q_0)$ with $q = q_0^2 \ge 9$, and either $q_0 = 3$ or $|G:G_0|$ is odd; or
        \item[{\rm (v)}] $(q,H) \in \{(5,A_4),(5,S_4), (7,S_4), (9,A_5), (9,S_5), (11,A_5), (19,A_5)\}$. 
	\end{itemize}
 Equivalently, $G$ is not semi-Frobenius if and only if either $b(G) = 2$, or $H$ is of type $\GL_2(q^{1/2})$ with $q \ge 16$ and $|G:G_0|$ even.
\end{thm}

\begin{proof}
    First note that if $b(G) = 2$, then $G$ is not Frobenius, and therefore not semi-Frobenius. Thus, we shall assume that $b(G) > 2$, and we shall use Theorem~\ref{t:PSL2_b(G)} without further reference to determine $b(G)$. We divide the proof into several cases, corresponding to the possible types of $H$. Throughout, we  write $q = p^f$, with $p$ prime.

    \vs
	
	\noindent \emph{Case 1. $H$ is of type $P_1$.}

    \vs

    Since $G$ is $2$-transitive, $\Sigma(G)$ is complete.

    \vs
	
	\noindent \emph{Case 2. $H$ is of type $\GL_1(q)\wr S_2$.}

    \vs

    Here, $b(G) > 2$ if and only if $\PGL_2(q) < G$, in which case $b(G) = 3$. Assume therefore that $\PGL_2(q) < G$. We may identify $\Omega$ 
    with the set of distinct pairs of $1$-dimensional subspaces of $\mathbb{F}_q^2$, and as shown in the proof of \cite[Lemma 4.7]{B_sol}, $\{\a,\b,\g\}$ is a base for $G$, where
    \begin{equation*}
        \a := \{\la e_1\ra, \la e_2\ra\},\ \b := \{\la e_1\ra,\la e_1+e_2\ra\}, \text{ and } \g := \{\la e_1\ra,\la e_1+\mu e_2\ra\},
    \end{equation*}
    with $\{e_1,e_2\}$ an arbitrary basis for $\mathbb{F}_q^2$ and $\mu$ a generator of $\mathbb{F}_q^\times$. Thus $\{\a,\b\}$ is an edge in $\Sigma(G)$. Moreover, for each $\lambda \in \mathbb{F}_q^\times$, replacing $e_2$ by $\lambda e_2$ shows that $\a$ and $\{\la e_1\ra,\la e_1+\lambda e_2\ra\}$ are adjacent in $\Sigma(G)$. By symmetry, $\a$ is also adjacent to $\{\la e_2\ra,\la e_1+\lambda e_2\ra\}$. It remains to prove that, for all distinct  $\lambda_1, \lambda_2 \in \mathbb{F}_q^\times$, there exists a base for $G$ containing both $\a$ and $\delta := \{\la e_1+\lambda_1e_2\ra,\la e_1+\lambda_2 e_2\ra\}$. 

    We claim that there exists $\nu \in \{\mu,\mu^{-1}\}$ such that $\lambda_1 \lambda_2^{-1} \ne \nu^{p^j-1}$ for all $j \in \{1,\ldots,f\}$. Otherwise, since $\lambda_1 \ne \lambda_2$, there would exist $j,k \in \{1,\ldots,f-1\}$ such that $\mu^{p^j-1} = (\mu^{-1})^{p^k-1}$, i.e. $\mu^{p^j+p^k-2} = 1$. Since $p^j+p^k-2 < 2q-2$, this would imply that $p^j+p^k-2=q-1$. However, $1 \not\equiv 2 \pmod p$, and so our claim follows. Now, let $\theta:=\{\la e_1\ra, \la e_1+\lambda_1 \nu e_2\ra\}$. We will show that $\{\a,\theta,\delta\}$ is a base for $G$.
    
    The pointwise stabiliser of $\{\a,\theta\}$ in $\Gamma \mathrm{L}_2(q)$ consists of all elements of the form $g_{j,x} := \sigma^j \mathrm{diag}((\lambda_1 \nu)^{p^j-1}x, x)$, where $1\leq j \leq f$, $x \in \mathbb{F}_q^\times$, and
 $\sigma$ is the field automorphism such that $(c_1e_1+c_2e_2)^\sigma = c_1^pe_1+c_2^pe_2$ for all $c_1,c_2 \in \mathbb{F}_q$. Additionally, $g_{j,x}$ maps $U:=\la e_1+\lambda_1 e_2 \ra$ to $\la e_1 +  \lambda_1 (\nu^{-1})^{(p^j-1)} e_2 \ra$. Suppose now that $g_{j,x}$ also stabilises $\delta$. If $g_{j,x}$ swaps the two subspaces in $\delta$, then it follows that $\lambda_1 \lambda_2^{-1} = \nu^{p^j-1}$, contradicting the definition of $\nu$. Hence $g_{j,x}$ fixes $U$, and so $(\nu^{-1})^{p^j-1} = 1$, yielding $j = f$ and $g_{j,x} \in Z(\mathrm{GL}_2(q))$. Therefore, $\{\a, \theta,\delta\}$ is a base for $G$.

\vs
	
	\noindent \emph{Case 3. $H$ is of type $\GL_1(q^2)$.}

    \vs

In this case, $b(G) > 2$ if and only if $\PGL_2(q)\leqs G$, in which case $b(G) = 3$. {\sc Magma} computations show that $G$ is semi-Frobenius when $q = 4$, so we shall assume that $q > 4$. As in the proof of \cite[Lemma 4.8]{B_sol}, we identify $G_0$ with the unitary group $R :=\mathrm{U}_2(q)$, and $\Omega$ with the set of orthogonal pairs of non-degenerate 1-dimensional subspaces of the natural module for $R$ over $\mathbb{F}_{q^2}$.
More explicitly, fixing an orthonormal basis $\alpha:=\{u,v\}$ for this module and defining $\a := \{\la u\ra,\la v\ra\}$, it follows that 
\[
\Omega = \{\alpha\} \cup \{\tau_\zeta : \zeta \in \mathbb{F}_{q^2}^\times, \zeta^{q+1} \neq -1 \},
\]
where $\tau_\zeta := \{\la u+\zeta v\ra , \la u-\zeta^{-q} v \ra \}$. Fix a generator $\mu$ of $\mathbb{F}_{q^2}^\times$, and let $\zeta \in \mathbb{F}_{q^2}^\times$ with $\zeta^{q+1} \ne -1$. We shall prove by contradiction that if $-\mu^{-(q-1)} \in X:=\{\zeta^2,-\zeta^{-(q-1)}\}$, then $-(\mu^{-1})^{-(q-1)} \notin X$. First suppose that $-\mu^{-(q-1)} = -(\mu^{-1})^{-(q-1)}$. Then $1 = \mu^{2(q-1)}$, contradicting the fact that $|\mu| = q^2-1 > 2(q-1)$. Without loss of generality, we may therefore assume that $\zeta^2 = -\mu^{-(q-1)}$. Since $-1 = \mu^{(q^2-1)/(2,q-1)}$ and $\zeta^{q+1} \neq -1$, we deduce that $q \equiv 1 \pmod 4$ and $\zeta = \pm \mu^{(q-1)^2/4}$. It follows that $\zeta^{-(q-1)} = \mu^{-(q-1)^3/4}$. If $-(\mu^{-1})^{-(q-1)}$ lies in $X$, then it is equal to $-\zeta^{-(q-1)}$, and so $1 = \mu^{q-1+(q-1)^3/4}$. Hence $q-1+(q-1)^3/4 = k(q^2-1)$ for some positive integer $k$, and solving this cubic equation shows that $\sqrt{k^2+2k-1}$ is an integer. However, $k^2+2k-1 = (k+1)^2-2$ is not a square, a contradiction. We have therefore proved our claim.

Now, choose $\nu \in \{\mu,\mu^{-1}\}$ such that $-\nu^{-(q-1)} \notin X$. 
We shall show that $\{\alpha, \tau_\nu, \tau_\zeta\}$ is a base for $G$. The group $\Gamma \mathrm{U}_2(q)$, generated by $\GU_2(q)$ and field automorphisms, satisfies $G \le \Aut(\UU_2(q)) \cong \Gamma \mathrm{U}_2(q)/Z(\GU_2(q))$. Following the proof of \cite[Lemma 4.8]{B_sol}, we see that (since $q > 4$) the pointwise stabiliser of $\{\alpha,\tau_\nu\}$ in $\Gamma \mathrm{U}_2(q)$ is generated by $Z(\mathrm{GU}_2(q))$ along with 
\[
M := \begin{pmatrix}
    0 & -\nu^{-(q-1)}\\
    1 & 0\\
\end{pmatrix}.
\]
Note that $M^2 \in Z(\mathrm{GU}_2(q))$. Hence it suffices to prove that $M$ does not fix $\tau_\zeta$. It is easy to see that if $M$ fixes each of the two subspaces in $\tau_\zeta$, then $-\nu^{-(q-1)} = \zeta^2$, and if $M$ swaps these subspaces, then $-\nu^{-(q-1)} = -\zeta^{-(q-1)}$. However, $-\nu^{-(q-1)} \notin X$, and so neither of these cases occurs.  Therefore, $M$ does not fix $\tau_\zeta$, and so $\{\alpha, \tau_\nu, \tau_\zeta\}$ is a base for $G$.
It follows immediately that $G$ is semi-Frobenius.

\vs

	\noindent \emph{Case 4. $H$ is a subfield subgroup of type $\GL_2(q_0)$, where $q = q_0^k$ for some prime $k\geqs 2$.}

    \vs

If $k\geqs 3$ then $b(G) = 2$, so we only need to consider the case where $k = 2$, and the maximality of $H$ implies that $G\leqs \PSigmaL_2(q)$. This case requires a more refined treatment. Note that if $q_0 = 2$, then $H$ is considered as a group of type $\GL_1(q)\wr S_2$, and if $q_0 = 3$, then we obtain the result via {\sc Magma} calculations. We shall therefore assume that $q_0 \ge 4$, so that $b(G) = 3$, and separately consider two subcases, depending on the parity of $|G:G_0|$.

\vs
	
	\noindent \emph{Subcase 4(a). $o:=|G:G_0|$ is odd.}

    \vs

Let $\a \in \Omega$ be such that $H = G_\a$. We first note that, up to conjugacy in $H_0 := H\cap G_0$, the point stabilisers of $H_0$ are as follows (see \cite[p.~354]{FI_subdeg}):
    \begin{equation*}
        H_0,\ C_p^{f/2},\ C_{q_0+1}\ \mbox{($(q_0-2)/2$ copies)},\ C_{q_0-1}\ \mbox{($q_0/2$ copies)}
    \end{equation*}
    if $q$ is even, and
    \begin{equation*}
            H_0,\ D_{2(q_0+\e)},\ C_p^{f/2},\ C_{q_0+1}\ \mbox{($(q_0-4-\e)/4$ copies)},\
            C_{q_0-1}\ \mbox{($(q_0-2+\e)/4$ copies)}
    \end{equation*}
    if $q$ is odd, where $\e = \pm 1$ satisfies $q_0\equiv \e\pmod 4$.
    Now, for each $\delta \in \{1,-1\}$, let $Z_\delta$ be the subgroup of $\mathbb{F}_q^\times$ of order $q_0+\delta$. As discussed in \cite[p.~354]{FI_subdeg}, the fusion of $G_0$-suborbits of length $q_0(q_0-\delta)$ (i.e.~with point stabiliser $C_{q_0+\delta}$) under a field automorphism of $G_0$ corresponds to the fusion under the corresponding field automorphism of $\mathbb{F}_q$ of the sets $\{x, x^{-1}\}$, for elements $x \in \mathbb{F}_q^\times/Z_\delta$ of order at least $3$. It follows that the point stabilisers of $H = G_\a$ are isomorphic to the following groups, for certain divisors $i_+$ and $i_-$ of $o$:
    \begin{equation*}
        H,\ C_p^{f/2}.o,\        
        C_{q_0+1}.i_+,\ C_{q_0-1}.i_-,\ \mbox{and (if $q$ is odd and $q_0\equiv \e\ (\mathrm{mod}\ 4))$} \
        D_{2(q_0+\e)}.o.
    \end{equation*}
    
We now prove that $C_{q_0 + 1}$ and $C_{q_0 - 1}$ always appear as point stabilisers. It suffices to show that the primitive element $\mu$ of $\mathbb{F}_q$ satisfies $|Z_\delta \mu^{\langle \sigma \rangle}| = o$, where $\sigma$ is the automorphism of $\mathbb{F}_q$ of odd order $o$. Let $q_1$ be such that $q_1^o=q$ and hence $\mu^{\langle \sigma\rangle} = \{\mu^{q_1^m}\mid 0\leq m \leq o-1\}$. Now $|Z_\delta \mu^{\langle \sigma \rangle}| = o$ if and only if $\mu^{q_1^m-q_1^l}\notin Z_\delta$ for all $0\leqs  \ell < m\leqs o-1$. Suppose for a contradiction that $\mu^{q_1^m-q_1^\ell}\in Z_\delta$ for such $\ell$ and $m$. Then $\mu^{(q_1^m-q_1^\ell)(q_0+\delta)} = 1$, and so $(q-1)\mid (q_1^m-q_1^\ell)(q_0+\delta)$. Hence $(q_0-\delta)\mid q_1^\ell(q_1^{m-\ell}-1)$, which implies $(q_0-\delta)\mid (q_1^{m-\ell}-1)$. It follows easily that either $(q_0-\delta) = (q_1^{m-\ell}-1)$ or $(q_0-\delta) \mid (q_1^{m-\ell}q_0^{-1}-\delta)$. However, since $o$ is odd and $q_1^{m-1}q_0^{-1} < q_0$, neither case can occur, a contradiction. 

Next, let $\beta \in \Omega \setminus \{\alpha\}$, and choose $\gamma \in \Omega$ so that (up to isomorphism) the ordered pair $(H_\beta,H_\gamma)$ is one of $(C_p^{f/2}.o,C_{q_0+1})$, $(D_{2(q_0 + \varepsilon)}.o,C_{q_0 - \varepsilon})$ and $(C_{q_0 \pm 1}.i,C_{q_0 \mp 1})$, with $i$ a divisor of $o$. Since $H_\gamma \le G_0$, we observe that $H_\beta \cap H_\gamma = (H_\beta \cap G_0) \cap H_\gamma$. In particular, if $H_\beta \cong C_p^{f/2}.o$, then $H_\beta \cap H_\gamma = 1$, and so $\{\alpha,\beta,\gamma\}$ is a base for $G$. Thus $\{\alpha, \beta\}$ is an edge in $\Sigma(G)$. In the remaining two cases, either $H_\beta \cap H_\gamma = 1$ and $\{\alpha, \beta\}$ is again an edge, or $H_\beta \cap H_\gamma = \langle g \rangle$, where $g$ is the unique involution of $H_\g$. As $H_\b$ is a core-free subgroup of $H$, there exists $h\in H$ such that $g\notin H_\b^h$, and hence $g\notin H_{\b^h}\cap H_\g$, which yields $H_{\b^h}\cap H_\g = 1$. This shows that $\{\a,\b,\g^{h^{-1}}\}$ is a base for $G$, and thus $\Sigma(G)$ is complete.

\vs
	
	\noindent \emph{Subcase 4(b). $|G:G_0|$ is even.}

    \vs

To show that $\Sigma(G)$ is not complete, it suffices to consider the case where $G$ is generated by $G_0$ and an involutory field automorphism. Similarly to the proof of Proposition~\ref{prop:basesizeq0}, we identify $G$ with the orthogonal group $\PGO^-_4(q_0)$, where $\mathrm{GO}^-_4(q_0)$ is the isometry group of a non-degenerate quadratic form $Q$ of minus type on $V:=\mathbb{F}_{q_0}^4$. As above, the action of $G$ on the right cosets of $H$ is equivalent to its action on the set $\Delta$ of one-dimensional subspaces $Y$ of $V$ such that there exists $y \in Y$ with $Q(y)$ a non-zero square in $\mathbb{F}_{q_0}$. We shall therefore complete the proof by identifying a pair of elements of $\Delta$ that does not extend to a base for $G$ of size $b(G) = 3$.

Assume first that $q_0$ is even, and let $\rho$ be the polar form of $Q$. As in the proof of Proposition~\ref{prop:basesizeq0}, there exists a basis $\{e_1,\ldots,e_4\}$ for $V$ so that the matrix of $Q$ is
\[
M_Q := \begin{pmatrix}
    0&0&0&1\\
    0 & 1 & 1& 0\\
    0&0& \zeta &0\\
    0&0&0&0\\
\end{pmatrix},
\]
for an arbitrary $\zeta \in \mathbb{F}_{q_0}^\times$ such that the polynomial $x^2+x+\zeta$ is irreducible over $\mathbb{F}_{q_0}$. Since $q_0 \ge 4$, there are at least two choices for $\zeta$, and so we will assume without loss of generality that $\zeta \ne 1$.

Now, let $v_1:=e_1+\zeta e_2+e_3+\zeta e_4$, $v_2:=e_2$, $v_3:=e_1$ and $v_4:=e_3$, so that $\{v_1,v_2,v_3,v_4\}$ is a basis for $V$. Then $Q(v_2) = 1$, and
\[
Q(v_1) = \zeta \rho(e_1,e_4)+ \zeta^2 Q(e_2) + \zeta \rho(e_2,e_3) + Q(e_3) = \zeta+\zeta^2+\zeta+\zeta = \zeta(1+\zeta) \ne 0.
\]
Hence $\langle v_1 \rangle$ and $\langle v_2 \rangle$ are distinct subspaces in $\Delta$. Additionally, let $u \in V \setminus \langle v_1, v_2 \rangle$, so that $u = \sum_{i=1}^4 \alpha_i v_i$ with $\alpha_i \in \mathbb{F}_{q_0}$ such that $\alpha_3$ and $\alpha_4$ are not both $0$. Finally, let $m:=\zeta \alpha_3^2 + \alpha_4(\alpha_3+\alpha_4)$ and
\[
A:= m^{-1}\begin{pmatrix}
    \zeta \alpha_3^2 & \zeta \alpha_3 \alpha_4 & 0 & \zeta \alpha_4^2\\
    0 & m & 0 & 0\\
    \alpha_3 (\alpha_3 + \alpha_4) & \zeta \alpha_3^2 & m & \zeta \alpha_3 \alpha_4\\
    \zeta^{-1}(\alpha_3+\alpha_4)^2 & \alpha_3(\alpha_3+\alpha_4) & 0  & \zeta \alpha_3^2\\
\end{pmatrix}.
\]
Note that if $\alpha_3 = 0$, then $m = \alpha_4^2 \ne 0$, and otherwise $m = (\alpha_3^2)((\alpha_4 \alpha_3^{-1})^2+\alpha_4 \alpha_3^{-1}+\zeta)$, which is again non-zero by the definition of $\zeta$. It is straightforward to check that $\det(A) = 1$, that the diagonal entries of $A M_Q A^T$ are equal to the corresponding diagonal entries of $M_Q$, and that the Gram matrix $M_\rho = M_Q + M_Q^T$ of the polar form $\rho$ satisfies $A M_\rho A^T = M_\rho$. Hence $A \in \mathrm{GO}^-_4(q_0)$. We also observe that $A$ is a non-scalar matrix that fixes $u$ and each vector in $\langle v_1, v_2 \rangle$. Since each subspace in $\Delta$ is spanned either by such a vector $u$ or by a vector in $\langle v_1, v_2 \rangle$, the subset $\{\langle v_1 \rangle, \langle v_2 \rangle\}$ of $\Delta$ does not extend to a base for $G$ of size $3$.

Next, suppose that $q_0$ is odd. We deduce from \cite[p.~45]{kleidmanliebeck} that there exists a basis $\{e_1,\ldots,e_4\}$ for $V$ so that the matrix of $Q$ is
\[
M_Q := \begin{pmatrix}
    0&1&0&0\\
    0 & 0 & 0& 0\\
    0&0&-1 &-r\\
    0&0&0& -s\\
\end{pmatrix},
\]
where $r:=\omega+\omega^{q_0}$ and $s:=\omega^{q_0+1}$ for an arbitrary element $\omega$ of $\mathbb{F}_q \setminus \mathbb{F}_{q_0}$. Note that $r$ and $s$ do indeed lie in $\mathbb{F}_{q_0}$, as they are each equal to their $q_0$-th power.

We now proceed similarly to above. Since $q_0 > 3$, there exists $y \in \mathbb{F}_{q_0}^\times \setminus \{\pm 1\}$. Let $v_1:=e_1+e_2$, $v_2:=e_1+y^2e_2$, $v_3:=e_3$ and $v_4:=e_4$, so that $\{v_1,v_2,v_3,v_4\}$ is a basis for $V$. As $Q(v_1) = 1$ and $Q(v_2) = y^2$ are both non-zero squares in $\mathbb{F}_{q_0}$, it follows that $\langle v_1 \rangle$ and $\langle v_2 \rangle$ are distinct subspaces in $\Delta$. Additionally, each $u \in V \setminus \langle v_1, v_2 \rangle$ satisfies $u = \sum_{i=1}^4 \alpha_i v_i$ for some $\alpha_i \in \mathbb{F}_{q_0}$ such that $\alpha_3$ and $\alpha_4$ are not both $0$. For such a vector $u$, let $t:=(\alpha_3+\alpha_4 \omega)(\alpha_3+\alpha_4 \omega^{q_0})$ and
\[
A:= t^{-1}\begin{pmatrix}
    t & 0 & 0 & 0\\
    0 & t & 0 & 0\\
    0 & 0 & \alpha_3^2 - s \alpha_4^2 & (2 \alpha_3 + r \alpha_4) \alpha_4\\
    0 & 0 & (r \alpha_3 + 2s \alpha_4) \alpha_3  & -\alpha_3^2 + s \alpha_4^2\\
\end{pmatrix}.
\]
Note that $t \in \mathbb{F}_{q_0}$ as $t^{q_0} = t$, and that $t \ne 0$ since $\alpha_3$ and $\alpha_4$ lie in $\mathbb{F}_{q_0}$, while $\omega$ and $\omega^{q_0}$ do not. As above, $A$ is a non-scalar matrix in $\mathrm{GO}^-_4(q_0)$ (this time with determinant $-1$) that fixes $u$ and each vector in $\langle v_1, v_2 \rangle$. We conclude that $\{\langle v_1 \rangle,\langle v_2 \rangle\}$ does not extend to a base for $G$ of size $3$, as required.

\vs

	\noindent \emph{Case 5. $H$ is of type $2^{1+2}_-.\Omega_2^-(2)$ or type $A_5$.}

    \vs

Here, $b(G) > 2$ if and only if $(q,H)$ is one of the cases listed in (v). In each of these special cases, one can easily check that $\Sigma(G)$ is complete with the aid of {\sc Magma}.
\end{proof}

If $\soc(G) = \LL_2(q)$ and $b(G) = 2$, then Conjecture \ref{conj:BG} is verified in \cite[Theorem 4.22]{BH_Saxl}, extending the earlier result \cite[Theorem 1.3]{CD_Saxl}. By Theorem~\ref{t:PSL2_semi}, in order to show that this still holds when $b(G) > 2$, it remains to prove that the conjecture is satified when the point stabiliser of $G$ is of type $\GL_2(q^{1/2})$, with $q \ge 16$ and $|G:G_0|$ even. We leave this as an open problem.

Finally, we classify the groups $G$ with socle $\LL_2(q)$ such that $\Sigma(G)$ is $G$-arc-transitive, excluding the aforementioned difficult case. The following result is stated as Theorem \ref{introthm:PSL2_arc} in Section \ref{s:intro}.

\begin{thm}
    \label{t:PSL2_arc}
    Let $G$ be a primitive group with socle $G_0 = \LL_2(q)$ and point stabiliser $H$, and if $H$ is of type $\GL_2(q^{1/2})$ with $q \ge 16$, then assume that $|G:G_0|$ is odd. Then $\Sigma(G)$ is $G$-arc-transitive if and only if one of the following holds:
    \begin{itemize}\addtolength{\itemsep}{0.2\baselineskip}
		\item[{\rm (i)}] $H$ is of type $P_1$;
        \item[{\rm (ii)}] $G$ is $2$-transitive and $(G,H) = (\LL_2(11),A_5)$, $(\PGammaL_2(8),D_{18}.3)$, $(\LL_2(7),S_4)$, $(S_6,S_5)$, $(A_6,A_5)$, $(S_5,S_4)$ or $(A_5,A_4)$.
		\item[{\rm (iii)}] $b(G) = 2$, $G = \PGL_2(q)$, $H$ has type $\GL_1(q)\wr S_2$, and $5 \ne q\geqs 4$; or
		\item[{\rm (iv)}] $b(G) = 2$ and $(G,H) = (\LL_2(29),A_5)$, $(\PGL_2(11),S_4)$, $(\mathrm{M}_{10},5{:}4)$  or $(A_5,S_3)$.
	\end{itemize}
    In particular, $\reg(G) = 1$ if and only if $G$ is one of the groups in parts (iii) and (iv), or $H$ is of type $P_1$ and $G$ is sharply $3$-transitive, or $(G,H) \in \{(S_6,S_5), (A_6,A_5),(S_5,S_4)\}$.
\end{thm}

\begin{proof}
    First assume that $b(G) > 2.$ By Theorem \ref{thm:PSL2_semi} (as $G$ does not satisfy Theorem \ref{thm:PSL2_semi}(ii)), $G$ is semi-Frobenius. Hence, in this setting, $\Sigma(G)$ is $G$-arc-transitive if and only if $G$ is $2$-transitive. It follows from Theorem~\ref{t:PSL2_b(G)} that either $H$ is of type $P_1$, or $(G,H)$ is one of the following pairs:
    \begin{equation*}
        (\LL_2(11),A_5),\  (\PGammaL_2(8),D_{18}.3),\ (\LL_2(7),S_4),\ (S_6,S_5),\ (A_6,A_5),\ (S_5,S_4),\ (A_5,A_4).
    \end{equation*}
    Also note that if $G$ is $2$-transitive, then $\reg(G) = 1$ if and only if $G$ is sharply $b(G)$-transitive.

    To complete the proof, we consider the case where $b(G) = 2$, so that $\Sigma(G)$ is $G$-arc-transitive if and only if $\reg(G) = 1$. Here either $H$ is soluble, or $H$ has type $A_5$ or $\GL_2(q_0)$ with $q_0^k$ for some prime $k\geqs 3$. For the former case, the result can be deduced from \cite[Theorem 1.6]{BH_Saxl}, and the relevant groups with $\reg(G) = 1$ are listed in the statement. Notice that if $Q(G,2) < 1/2$ and $2|H|^2 < |G|$, then $\reg(G)\geqs 2$ and thus $\Sigma(G)$ is not $G$-arc-transitive. The bound $2|H|^2 < |G|$ always holds if $H$ is of type $\GL_2(q_0)$, and as discussed in the proof of \cite[Theorem 4.22]{BH_Saxl} we see that $Q(G,2) < 1/2$. For the case where $H$ has type $A_5$, we can check using {\sc Magma} that $\reg(G) \geqs 2$ if $q<197$, unless $q = 29$ and $\reg(G) = 1$. If instead $q \geqs 197$, then the proof of \cite[Theorem 4.22]{BH_Saxl} shows that $Q(G,2) < 1/2$, and since $2|H|^2<|G|$ in this case, we obtain $\reg(G)\geqs 2$.
\end{proof}

\subsection{Almost simple groups with soluble point stabilisers}
The following theorem extends \cite[Theorem 1.1]{BH_Saxl}.

\begin{thm}
	\label{t:as_common_solstab}
	Let $G$ be an almost simple primitive group with soluble point stabilisers. Then $G$ satisfies Conjecture~\ref{conj:BG}.
\end{thm}

\begin{proof}
	The precise base size $b(G)$ is computed in \cite[Theorem 2]{B_sol}. In particular, we have $b(G)\leqs 5$ in every case. If $b(G) = 2$ then the result follows from \cite[Theorem 1.1]{BH_Saxl}. Thus, to complete the proof, we may  assume that $3\leqs b(G)\leqs 5$. Hence by \cite[Theorem 2]{B_sol}, $G$ appears among the infinite families and special cases listed in \cite[Tables 4--7]{B_sol}.
	
	Note that if $\soc(G) \in\{\LL_2(q),\UU_3(q),{^2}B_2(q),{^2}G_2(q)\}$ and the point stabiliser is of type $P_1$, then $G$ is $2$-transitive and so $\Sigma(G)$ is complete by Lemma \ref{l:pre_complete_2-trans}. Upper bounds for $\widehat{Q}(G,b(G))$ in the remaining infinite families can be found in \cite{B_sol} and the proof of \cite[Lemma 3.7]{BH_prod}. In almost all cases, we deduce that $\widehat{Q}(G,b(G)) < 1/2$, and it follows (see the end of \S\ref{s:pre}) that any two vertices in $\Sigma(G)$ have a common neighbour, i.e.~$G$ satisfies Conjecture~\ref{conj:BG}. Any remaining group in these infinite families that may satisfy $\widehat{Q}(G,b(G))\geqs 1/2$ satisfies one of the following (here $G_0 = \soc(G)$ and $H$ is a point stabiliser):
	\begin{itemize}\addtolength{\itemsep}{0.2\baselineskip}
		\item[{\rm (a)}] $G_0 = \LL_3(q)$, $H$ is of type $P_{1,2}$, and either $3 \le q \le 16$, $19 \le q \le 27$, or $q\in\{32,47,64\}$;
		\item[{\rm (b)}] $G_0 = \Sp_4(q)$, $H\cap G_0 = [q^4] \sd C_{q-1}^2$, and $q\in\{4,8,32\}$; or
		\item[{\rm (c)}] $G_0 = \LL_2(q)$, $H$ is of type $\GL_1(q^2)$, and $11\leqs q\leqs 16$.
	\end{itemize}
 
	Combining probabilistic and computational methods (similar to those mentioned in the proofs of Proposition~\ref{prop:sporsemi} and Theorem~\ref{thm:sporcnc}), one can check that in each case in (a)--(c), and each special case appearing in \cite[Tables 4--7]{B_sol}, $G$ satisfies Conjecture~\ref{conj:BG}.
\end{proof}
\section{Affine groups}

\label{s:affine}
We now turn our attention to primitive groups of affine type. Let $V$ be an $n$-dimensional vector space over $\mathbb{F}_q$. A primitive affine type group can be written as $G=VH \leq \mathrm{AGL}_n(q)$, with $\Omega = V$, where the point stabiliser $H$ acts irreducibly on $V$.

Thus far, few authors have considered the Saxl graphs of these affine type groups.
Lee and Popiel \cite{LP_Saxl} proved one of the first results on these graphs, namely that Conjecture~\ref{conj:BG} holds for most primitive affine groups $G$ of base size two such that $H$ is an almost quasisimple group with $\soc(H/Z(H))$ a sporadic simple group. One of the tools used in their analysis was an equivalence between that conjecture and an arithmetic property of the underlying vector space. We begin this section by proving that an analogous equivalence holds for generalised Saxl graphs.

\begin{lem}
Let $G=VH$ be a primitive group of affine type with point stabiliser $H$, such that $b(G) \ge 2$. Then $\Sigma(G)$ satisfies Conjecture~\ref{conj:BG} if and only if every vector in $V$ can be written as the sum of two vectors lying in bases for $H$ of size $b(H)$.
\end{lem}

\begin{proof}
The proof proceeds similarly to that of \cite[Lemma 2.3]{LP_Saxl}. First note that $b(G) = b(H)+1$, and moreover, every base of minimal size for $G$ gives rise to a base of minimal size for $H$ and vice versa, by removing or adding $0$, respectively. Suppose that $G$ satisfies Conjecture~\ref{conj:BG}, and let $v \in V$. Then $v\in V$ and $0 \in V$ share a common neighbour, say $v_1$, in $\Sigma(G)$. Therefore, $\{0,v_1\}$ and $\{v_1,v\}$ are subsets of bases for $G$ of size $b(G)$. Since the image of a base under $G$ is also a base, the image $\{0,v-v_1\}$ of $\{v_1,v\}$ under the translation $-v_1$ must also be a subset of a base of minimal size. Since $H$ is the stabiliser of $0$ in $G$, both $v_1$ and $v-v_1$ lie in bases for $H$ of size $b(H)$, and we can write $v = v_1+(v-v_1)$. 

Conversely, suppose that each vector $v \in V$ can be written as the sum of two vectors in bases for $H$ of size $b(H)$, say $v=v_1+v_2$. Since $G$ is transitive, it suffices to show that $v$ and $0$ share a common neighbour in $\Sigma(G)$. The definitions of $v_1$ and $v_2$ imply that $\{0,v_1\}$ and $\{0,v_2\}$ are subsets of bases for $G$ of size $b(G)$. Applying the translation $v_1$ to the latter subset, we see that the same is true for $\{v_1,v_1+v_2\} = \{v_1,v\}$. Therefore, $v$ and $0$ share the common neighbour $v_1$, and the result follows.
\end{proof}
Our next result describes the valency of $\Sigma(G)$.
\begin{prop}\label{aff:almost-reg}
Let $G = VH$ be a primitive group of affine type, and let $\{\mathcal{O}_1, \dots, \mathcal{O}_k\}$ be the set of orbits of $H$ on $V$ that meet a base for $G$ of size $b(G)$ containing $0$.
Then
\[
\mathrm{val}(G) = \sum_{i=1}^k |\mathcal{O}_i|.
\]
\end{prop}
\begin{proof}
This follows directly from the fact that the image of a base is also a base.
\end{proof}

%In other words, the $H$-orbits $\mathcal{O}_1,\dots,\mathcal{O}_k$ are the \textit{almost-regular} $H$-orbits, which we will define later in Definition \ref{def:almost-reg}.

In the remainder of this section, we relate the investigation of $\Sigma(G)$ to other well-studied properties of groups, and use these to 
verify Conjecture~\ref{conj:BG} in several cases. The conjecture remains open for affine groups in general.

Recall from Lemma \ref{lem:orbgraphs} that $\Sigma(G)$ is a union of orbital graphs. In some cases, we can make inferences about properties of $\Sigma(G)$ from information about the orbital diameter of $G$. For instance, recall from Lemma~\ref{orbdiam2} that if $b(G) \ge 3$ and $G$ has orbital diameter two, then $\Sigma(G)$ satisfies Conjecture~\ref{conj:BG}. Similarly, if a group has orbital diameter 1, which occurs if and only if $G$ is 2-homogeneous, then $\Sigma(G)$ is complete. Of course, the converses of these statements do not hold in general.

Now, suppose that $G$ is such that $H$ is almost quasisimple. Such groups $G$ with orbital diameter two were considered in \cite{kalathesis}. As shown in \cite[Lemma 2.4.4]{kalathesis2}, a natural upper bound on the orbital diameter of $G$ is $\mathrm{rank}(G)-1$, where $\mathrm{rank}(G)$ is the permutation rank of $G$, i.e.~the number of orbitals of $G$ (note also that $\mathrm{rank}(G)-1$ is the number of distinct orbital graphs of $G$). Hence, each rank three primitive permutation group has orbital diameter two, and so Conjecture~\ref{conj:BG} holds. Such groups were classified in \cite{liebeckrank3}. Since in this case there are just two orbital graphs, either the generalised Saxl graph is complete, or it is a single orbital graph (and therefore $G$-arc-transitive with diameter two). We now provide an example of a group for which the latter occurs.
%Since in this case there are just two orbital graphs, one might expect that the generalised Saxl graph will be complete. However, it turns out that this is not always the case. 

\begin{prop}
    There exists a rank three primitive group $G$ of affine type such that $\Sigma(G)$ is a single orbital graph. 
\end{prop}

\begin{proof}
Let $G=VH$ be the primitive affine type group such that $H=3.A_6$ and $V=\mathbb{F}_4^3$. We can show using {\sc Magma} that $b(G)=3$, and that $\Sigma(G)$ is not complete. By \cite{liebeckrank3}, $G$ is a rank three group, and so it has two orbital graphs. Thus if $\Sigma(G)$ were a union of both, then it would be complete. The result follows.
\end{proof}

Note that there are also examples of rank three affine type groups $G$ such that $\Sigma(G)$ is complete, for example when $H = \mathrm{SL}_2(5) \cong 2.A_5$ and $V=\mathbb{F}_9^2$, so that $b(G) = 3$.
Further examples of primitive affine type groups that have orbital diameter two, and hence satisfy Conjecture~\ref{conj:BG}, are listed in \cite[Theorem 1.5 \& Example 4.2]{kalathesis} and \cite[Remark 6.3.2]{kalathesis2}. We list some of them with rank at least four in Table \ref{tab:affine2}. 

\begin{table}[ht]
   
\caption{Examples of primitive affine groups $G = VH$ with $\mathrm{diam}(\Sigma(G)) =2$. Here, $r:=\mathrm{rank}(G)-1$, and $V$ is defined over $\mathbb{F}_{l}$.}
    \centering
 \begin{tabular}{ccccc}
 \hline
 $H$ &$\dim(V)$&$l$&$r$& extra conditions\\
 \hline
   $G_2(q)$&$7$&$q$&$4$& $q$ odd\\
   $\Aut(G_2(3))$&$14$&$2$&$5$&\\
    $\Omega_9(q)$ &$16$&$q$&$4$&\\
    $3.A_7$&$3$&$5$&$11$&\\\hline
 \label{tab:affine2}
 \end{tabular}
\end{table}
\section{Primitive wreath products}

\label{s:wr}

Let $L\leqs\mathrm{Sym}(\Gamma)$ be a primitive group, and let $P\leqs S_k$ be a transitive group on $[k]$. Then $G:=L\wr P$ has a faithful primitive action on the Cartesian product $\Omega:=\Gamma^k$. More precisely, we have
\begin{equation*}
(\gamma_1, \ldots, \gamma_k)^{(z_1, \ldots, z_k) \sigma}=\left(\gamma_{1^{\sigma^{-1}}}^{z_{1^{\sigma-1}}}, \ldots, \gamma_{k^{k^{-1}}}^{z_{k^{\sigma-1}}}\right)
\end{equation*}
for any $\gamma_i\in\Gamma$, $z_i\in L$ and $\sigma\in P$. The group $G$ is called a \emph{primitive wreath product} of $L$ by $P$. As a special case, $G$ is a product type primitive group if $k\geqs 2$ and $L$ is either almost simple or of diagonal type.

Let $\reg(L,m)$ be the number of regular $L$-orbits on $\Gamma^m$ in its coordinatewise action, and let $\Pi := \{\pi_1,\dots,\pi_m\}$ be a partition of $[k]$ into $m$ parts. Then $\Pi$ is called a \emph{distinguishing partition} if its stabiliser $\bigcap_{i=1}^m G_{\{\pi_i\}}$ is trivial, noting that any refinement of a distinguishing partition is also distinguishing. Thus, it is natural to consider the minimal number $m$ such that there exists a distinguishing partition of $[k]$ into $m$ parts. We call this the \emph{distinguishing number} of $P$, and denote it $D(P)$. This value is closely related to $b(G)$, as follows.
\begin{thm}{\cite[Theorem 2.13]{BC_prod}}
	\label{t:prod_BC}
	Let $G = L\wr P$ be a primitive wreath product. Then $b(G)$ is the smallest integer $m$ such that $\reg(L,m) \geqs D(P)$.
\end{thm}

We now describe the bases for $G = L\wr P$ of size $b:=b(G)$. Let $A := \{\a_{i,j}\}$ be a $k\times b$ array, where $\a_{i,j}\in\Gamma$. Define a partition $\Pi$ of $[k]$ such that $i$ and $j$ are in the same part if and only if $(\a_{i,1},\dots,\a_{i,b})$ and $(\a_{j,1},\dots,\a_{j,b})$ are in the same $L$-orbit. Note that a column $(\a_{1,j},\dots,\a_{k,j})$ of $A$ lies in $\Omega = \Gamma^k$. Let $\Delta$ be the set of columns of $A$, so that $\Delta\subseteq\Omega$.

\begin{lem}
	\label{l:prod_bases}
	The set $\Delta$ is a base for $G$ if and only if each $\{\a_{i,1},\dots,\a_{i,b}\}$ is a base for $L$ and $\Pi$ is a distinguishing partition for $P$ into $D(P)$ parts.
\end{lem}

\begin{proof}
	Here we extend the argument in the proof of \cite[Lemma 4.1]{BH_prod}. First assume that each $\{\a_{i,1},\dots,\a_{i,b}\}$ is a base for $L$ and $\Pi$ is a distinguishing partition for $P$. Suppose that $x\in G_{(\Delta)}$, so that $x = z\sigma$ for some $z = (z_1,\dots,z_k)\in L^k$ and $\sigma\in P$. If $i^\sigma = j$, then $(\a_{i,1},\dots,\a_{i,b})^{z_i} = (\a_{j,1},\dots,\a_{j,b})$, which implies that $\sigma$ fixes the partition $\Pi$. Thus, $\sigma = 1$ as $\Pi$ is a distinguishing partition. Now $z_i\in \bigcap_{m = 1}^b G_{\a_{i,m}} = 1$ since $\{\a_{i,1},\dots,\a_{i,b}\}$ is a base for $L$. Therefore, $z = 1$, and so $\Delta$ is a base for $G$.
	
	To complete the proof, we assume that $\Delta$ is a base for $G$. Suppose for a contradiction that $\{\a_{i,1},\dots,\a_{i,b}\}$ is not a base for $L$, for some $i\in[k]$. Then there exists a non-identity element $x\in L$ fixing $\{\a_{i,1},\dots,\a_{i,b}\}$ pointwise. Now, let $z := (z_1,\dots,z_k)\in L$, where $z_i := x$ and $z_j := 1$ for $j\ne i$. It is then easy to show that $1\ne z\in G_{(\Delta)}$, a contradiction. Thus $\{\a_{i,1},\dots,\a_{i,b}\}$ is a base for $L$.
	
	Finally, we prove that $\Pi$ is a distinguishing partition for $P$. Recall that $\Pi$ is defined in terms of the $L$-orbits on $\Gamma^b$, so without loss of generality, we may assume that $\a_{i,m} = \a_{j,m}$ for any $m\in [k]$ if $i$ and $j$ are in the same part of $\Pi$, noting that the columns of $A$ still form a base for $G$ with this assumption. In particular, we have
	\begin{equation*}
	\mbox{$(\a_{i,1},\dots,\a_{i,b}) = (\a_{j,1},\dots,\a_{j,b})$ if and only if $i$ and $j$ are in the same part of $\Pi$.}
	\end{equation*}
	Thus, if $\sigma\in P$ fixes the partition $\Pi$, then it also fixes every column of $A$. Since the set $\Delta$ of columns of $A$ is a base for $G$, we obtain $\sigma = 1$, as required.
\end{proof}

Let $\mathcal{P}_m([k])$ be the set of ordered partitions of $[k]$ into $m$ parts, where some parts are allowed to be empty. Then $P$ has a natural action on $\mathcal{P}_m([k])$. Note that a partition $\Pi\in\mathcal{P}_m([k])$ is a distinguishing partition if and only if it lies in a regular $P$-orbit. Recall also that $\reg(G)$ denotes the number of regular orbits of $G$ in its componentwise action on $\Omega^{b(G)}$.

\begin{lem}
	\label{l:prod_r=1}
	Let $G = L\wr P$ be a primitive wreath product. Then $\reg(G) = 1$ if and only if $\reg(L,b(G)) = D(P)$ and $P$ has a unique regular orbit on $\mathcal{P}_{D(P)}([k])$.
\end{lem}

\begin{proof}
	Let $b = b(G)$ and let $A_1$ and $A_2$ be two $k\times b$ arrays of elements of $\Gamma$. For each $i \in \{1,2\}$, write $\Pi_i$ for the associated partition of $[k]$ with respect to $L$-orbits of rows of $A_i$, and write $\Delta_i$ for the set of columns of $A_i$. By Lemma \ref{l:prod_bases}, $\Delta_i$ is a base for $G$ if and only if $\Pi_i$ is a distinguishing partition and each row of $A_i$ is a base for $L$.
	
	First assume that $\reg(G) = 1$. Note that if $\Pi_1$ and $\Pi_2$ are in distinct $P$-orbits, then $\Delta_1$ and $\Delta_2$ are in distinct $G$-orbits. It follows by Lemma \ref{l:prod_bases} that $\reg(L,b(G)) = D(P)$ and $P$ has a unique regular orbit on $\mathcal{P}_{D(P)}([k])$.
	
	The other direction is clear by applying Lemma \ref{l:prod_bases}.
\end{proof}

Next, let us extend \cite[Theorem 6]{BH_prod} by determining primitive wreath products $G = L\wr P$ with $\reg(G) = 1$, where $P\leqs S_k$ is primitive, noting that \cite[Theorem 6]{BH_prod} only treats the case where $b(G) = 2$.

\begin{thm}
	\label{t:prod_reg(G)=1}
	Let $G = L\wr P$ be a primitive wreath product, where $P\leqs S_k$ is primitive. Then $\reg(G) = 1$ if and only if $\reg(L,b(G)) = D(P)$ and one of the following holds:
	\begin{itemize}\addtolength{\itemsep}{0.2\baselineskip}
		\item[{\rm (i)}] $P = S_k$ and $D(P) = k$;
		\item[{\rm (ii)}] $P = A_5$, $k = 6$ and $D(P) = 3$;
		\item[{\rm (iii)}] $P = \PGammaL_2(8)$, $k = 9$ and $D(P) = 3$; or
		\item[{\rm (iv)}] $P = \AGL_3(2)$, $k = 8$ and $D(P) = 4$.
	\end{itemize}
\end{thm}

\begin{proof}
	Note that \cite[Proposition 4.19]{BH_prod} classifies the primitive groups $P\leqs S_k$ with a unique regular orbit on $\mathcal{P}_{D(P)}([k])$. More specifically, $P$ is such a group if and only if $P = S_k$ or
	\begin{equation*}
	(P,k,D(P)) \in\{(A_5,6,3),(\PGammaL_2(8),9,3),(\AGL_3(2),8,4)\}.
	\end{equation*}
	Therefore, the theorem follows from Lemma \ref{l:prod_r=1}.
\end{proof}

To conclude this section, we give an application of Lemma \ref{l:prod_bases} to the study of Conjecture~\ref{conj:BG} for general primitive groups.  Let $G\leqs\mathrm{Sym}(\Omega)$ be transitive. We introduce the following terminology in order to generalise the notion of regular suborbits (that is, regular orbits of $G_\alpha$ for $\alpha \in \Omega$) to the case $b(G) \ge 3$.

\begin{defn}
	\label{def:almost-reg}
	Let $\alpha \in \Omega$. An orbit $\Delta\ne\{\a\}$ of $G_\a$ on $\Omega$ is called \emph{almost-regular} if $\Delta\cap B\ne \varnothing$ for some base $B$ for $G$ of size $b(G)$ with $\a\in B$.
\end{defn}

In other words, $\beta$ lies in an almost-regular $G_\a$-orbit if $\{\a,\b\}$ can be extended to a base for $G$ of size $b(G)$, i.e.~if $\{\a,\b\}$ is an edge in $\Sigma(G)$. Thus, the set of neighbours of $\a$ in $\Sigma(G)$ is equal to the union of almost-regular $G_\a$-orbits. For example, for a primitive affine group $G = VH$, the $H$-orbits $\mathcal{O}_1,\dots,\mathcal{O}_k$ from Proposition~\ref{aff:almost-reg} are precisely the almost-regular $H$-orbits. Note that any permutation group $G$ with $b(G) \ge 2$ has an almost-regular suborbit, and if $b(G) = 2$, then a suborbit is regular if and only if it is almost-regular.
The set $N(\a)$ of neighbours of $\a$ in $\Sigma(G)$ is exactly the union of almost-regular $G_\a$-orbits. In particular, if $b(G) = 2$, then $N(\a)$ is the union of regular $G_\a$-orbits, and it is conjectured by Burness and Huang \cite[Conjecture 8]{BH_prod} that if $G$ is primitive, then for any $\b\in\Omega$, the neighbourhood $N(\beta)$ meets every regular $G_\a$-orbit. In the same paper, they show that this conjecture is equivalent to the base-two version of Conjecture~\ref{conj:BG}. Here, we give a similar extension as follows.

\begin{con}
	\label{conj:strong}
	Let $G\leqs\mathrm{Sym}(\Omega)$ be a finite primitive permutation group with $b(G)\geqs 2$. Then for any $\a,\b\in\Omega$, the $\Sigma(G)$-neighbourhood $N(\b)$ meets every almost-regular $G_\a$-orbit.
\end{con}

\begin{thm}
	\label{t:conj:strong_equiv}
	Conjecture \ref{conj:strong} is equivalent to Conjecture \ref{conj:BG}.
\end{thm}

\begin{proof}
	Clearly, Conjecture \ref{conj:strong} implies Conjecture \ref{conj:BG}. Let $L\leqs\mathrm{Sym}(\Gamma)$ be a counter-example for Conjecture \ref{conj:strong}. That is, there exist vertices $\a,\b\in \Sigma(L)$ such that $N(\b)\cap \Delta = \varnothing$ for some almost-regular $L_\a$-orbit $\Delta$. Set $r := \reg(L)$, and let $G = L\wr S_r$ act on $\Omega = \Gamma^r$ with its product action. Then $G$ is primitive, and by Theorem \ref{t:prod_BC}, we have $b(G) = b(L)$. It suffices to show that the two vertices $(\a,\dots,\a)$ and $(\b,\dots,\b)$ in $\Sigma(G)$ have no common neighbour.
	
	Let $b:=b(G)$, and suppose that the set
	\begin{equation*}
	\{(\a,\dots,\a),(\a_{1,2},\dots,\a_{r,2}),(\a_{1,3},\dots,\a_{r,3})\dots,(\a_{1,b},\dots,\a_{r,b})\}
	\end{equation*}
	is a base for $G$. Then by Lemma \ref{l:prod_bases}, each $\{\a,\a_{i,2},\a_{i,3},\dots,\a_{i,b}\}$ is a base for $L$, and for any $i\ne j$, the $b$-tuples $(\a,\a_{i,2},\a_{i,3},\dots,\a_{i,b})$ and $(\a,\a_{j,2},\a_{j,3},\dots,\a_{j,b})$ are in distinct $L$-orbits. Since $r = \reg(L)$, there are at most $r$ almost-regular $L_\a$-orbits, so each set $\{\a_{1,i},\dots,\a_{r,i}\}$ meets every almost-regular $L_\a$-orbit. In particular, for each $2\leqs i\leqs b$, there exists $j\in[r]$ such that $\a_{j,i}\in\Delta$, and so our assumption implies that $\a_{j,i}\notin N(\beta)$. By applying Lemma \ref{l:prod_bases} once again, we see that, for each $i$, the point $(\a_{1,i},\dots,\a_{r,i})\in \Omega$ is not in any almost-regular orbit of $G_{(\b,\dots,\b)}$. Therefore, the two vertices $(\a,\dots,\a)$ and $(\b,\dots,\b)$ in $\Sigma(G)$ have no common neighbour. This completes the proof.
\end{proof}

\section{Another generalisation: the irredundant base graph}

\label{s:irredundant}

In this final section, we briefly explore an alternative interesting generalisation of the Saxl graph. Recall that an ordered subset $ (\alpha_1,\dots,\alpha_k)$ of $\Omega$ is an \emph{irredundant base} for $G$ if
\begin{equation*}
G>G_{\alpha_1}>G_{\alpha_1,\alpha_2}>\cdots >G_{\alpha_1,\dots,\alpha_{k-1}}>G_{\alpha_1,\dots,\alpha_k} = 1.
\end{equation*}
In particular, any irredundant base is an ordered base for $G$, and any ordered minimal base for $G$ is an irredundant base.

Now, let $I(G)$ be the maximal size of an irredundant base. It is clear that $b(G)\leqs I(G)$. For an integer $k \ge 2$, we define $I\Sigma_k(G)$ to be the graph with vertex set $\Omega$, such that two distinct vertices adjacent if and only if they lie in a common irredundant base for $G$ of size $k$. By a result of Cameron, published as \cite[Theorem 1.1]{dvms}, $I\Sigma_k(G)$ is non-empty if and only if $b(G)\leqs k\leqs I(G)$. Note also that $I\Sigma_k(G)$ is not necessarily a subgraph of $I\Sigma_k(L)$ if $L<G$. The following observations are immediate consequences of the definition of $I\Sigma_k(G)$. 

\begin{lem}
	\label{l:ISigma_k_observations}
	Let $G$ be a permutation group and $k$ an integer at least two.
	\begin{itemize}
		\item[\rm (i)] If $G$ is transitive, then $I\Sigma_k(G)$ is $G$-vertex-transitive.
		\item[\rm (ii)] If $G$ is primitive and $I\Sigma_k(G)$ is non-empty, then $I\Sigma_k(G)$ is connected.
		
		\item[\rm (iii)] $I\Sigma_{b(G)}(G) = \Sigma(G)$.
		
	\end{itemize}
\end{lem}

It is natural to drop the restriction on the size of the irredundant bases under consideration, as follows. For a group $G$ with $b(G) \ge 2$, let
\begin{equation*}
I\Sigma(G) := \bigcup_{k=b(G)}^{I(G)}I\Sigma_k(G).
\end{equation*}
Clearly,  $\Sigma(G)$ is a subgraph of $I\Sigma(G)$. It is also easy to see that $I\Sigma(G)$ is $G$-vertex-transitive if $G$ is transitive, and connected if $G$ is primitive. 

We now determine precisely when the graph $I\Sigma(G)$ is complete, and show that for all finite $G$, this graph satisfies an analogue of Conjecture~\ref{conj:BG}, i.e.~any two vertices have a common neighbour.

\begin{lem}\label{gagb}
Let $G\leq \mathrm{Sym}(\Omega)$ be finite. Then two distinct elements $\a, \b \in \Omega$ are adjacent in $I\Sigma(G)$ if and only if $G_\a\neq G_\b$. Hence $I\Sigma(G)$ is complete if and only if there are no pairs $\a,\b \in \Omega$ such that $G_\a=G_\b.$
\end{lem}

\begin{proof}
Let $\a,\b \in \Omega$. First assume that $G_\a=G_\b$, and suppose for a contradiction that some irredundant base for $G$ contains both $\a$ and $\b$. Then without loss of generality, this irredundant base is of the form $(x_1,\dots,x_l,\a,x_{l+1},\dots,x_k,\b,x_{k+1},\dots,x_n)$. Hence by the definition of an irredundant base, \[G_{x_1,\dots,x_l,\a,x_{l+1},\dots,x_k}>G_{x_1,\dots,x_l,\a,x_{l+1},\dots,x_k,\b}.\] However, $G_{x_1,\dots,x_l,\a,x_{l+1},\dots,x_k}\leq G_{\a}=G_{\b}$, and so in fact \[G_{x_1,\dots,x_l,\a,x_{l+1},\dots,x_k}=G_{x_1,\dots,x_l,\a,x_{l+1},\dots,x_k,\b},\] a contradiction. If instead $G_\a \ne G_\b$, then since $G$ is faithful, there exists an irredundant base containing $\a$ and $\b$. The result follows.
\end{proof}

\begin{cor}
If $G$ is primitive, then $I\Sigma(G)$ is complete.
\end{cor}
\begin{proof}
It is easy to see that the sets of elements of $\Omega$ with common point stabilisers form a system of imprimitivity. Therefore, if $G$ is primitive, then no two points $\a,\b \in \Omega$ satisfy $G_\a=G_\b.$ The result now follows from Lemma \ref{gagb}.
\end{proof}
\begin{cor}
Let $G \leq \mathrm{Sym}(\Omega)$ be finite. Then $I\Sigma(G)$ is either complete or a complete multipartite graph. 
In particular, any two vertices of $I\Sigma(G)$ have a common neighbour.  
\end{cor}
\begin{proof}
We may assume that $I\Sigma(G)$ is not complete. We observe from Lemma \ref{gagb} that $I\Sigma(G)$ is a complete multipartite graph, whose parts are the sets of elements of $\Omega$ with common point stabilisers. It follows immediately that any two vertices of $I\Sigma(G)$ have a common neighbour.
\end{proof}

\end{document}